\documentclass[12pt,letterpaper,final,twoside,leqno]{amsart}
\usepackage{amsmath,amssymb,amsthm,amsfonts,amscd,amsopn}
\usepackage{hyperref}
\usepackage[usenames]{color}
\usepackage{eucal, mathrsfs}
\usepackage{amsmath,amssymb,amsthm,%
amsfonts,
amscd,amsopn}
\usepackage[OT1,T1]{fontenc}
\usepackage{hyperref}
\usepackage{times}
\usepackage{xspace}
\usepackage{
color}
\usepackage[all]{xy}\xyoption{dvips}
\usepackage{comment} 
\usepackage{paralist}
\usepackage{stmaryrd}
\usepackage{enumitem}
\usepackage{supertabular}
\newcounter{are-there-sections}
\newcommand{\uj}[1]{{%
#1%
}}
\newcommand{\ujj}[1]{{%
}}

\newcommand{\ujjj}[1]{{%
#1%
}}

\def\me{S\'ANDOR J KOV\'ACS\xspace}

\def\mythanks{Supported in part by NSF Grant 
 DMS-0856185, and the
Craig McKibben and Sarah Merner Endowed Professorship in Mathematics at the
University of Washington.}

\def\myaddress{University of Washington, Department of Mathematics, Box 354350,
Seattle, WA 98195-4350, USA}

\def\myemail{skovacs@uw.edu\xspace}

\def\myurladdr{http://www.math.washington.edu/$\sim$kovacs\xspace}

\DeclareMathAlphabet{\smallchanc}{OT1}{pzc}%
                                 {m}{it}
\DeclareFontFamily{OT1}{pzc}{}
\DeclareFontShape{OT1}{pzc}{m}{it}%
             {<-> s * [1.100] pzcmi7t}{}
\DeclareMathAlphabet{\mathchanc}{OT1}{pzc}%
                                 {m}{it}

\newcommand{\mcH}{\mathchanc{H}}

\newcommand{\mcR}{\mathchanc{R}}

\newcommand{\mcm}{\mathchanc{m}}

\newcommand{\mco}{\mathchanc{o}}

\DeclareFontFamily{OMS}{rsfs}{\skewchar\font'60}
\DeclareFontShape{OMS}{rsfs}{m}{n}{<-5>rsfs5 <5-7>rsfs7 <7->rsfs10 }{}
\DeclareSymbolFont{rsfs}{OMS}{rsfs}{m}{n}
\DeclareSymbolFontAlphabet{\scr}{rsfs}

\newcommand{\sE}{\scr{E}}
\newcommand{\sF}{\scr{F}}

\newcommand{\sH}{\scr{H}}
\newcommand{\sI}{\scr{I}}

\newcommand{\sO}{\scr{O}}

\newcommand{\sfA}{{\sf A}}
\newcommand{\sfB}{{\sf B}}

\newcommand{\sfQ}{{\sf Q}}

\newcommand{\bC}{\mathbb{C}}

\newcommand{\bH}{\mathbb{H}}

\newcommand{\bN}{\mathbb{N}}

\newcommand{\bQ}{\mathbb{Q}}

\newcommand{\bZ}{\mathbb{Z}}

\newcommand{\ol}{\overline}

\newcommand{\ul}{\underline}

\newcommand{\into}{\hookrightarrow}

\newcommand{\wt}{\widetilde}
\newcommand{\what}{\widehat}

\newcommand{\rdown}[1]{\left\lfloor{#1}\right\rfloor}

\newcommand{\leteq}{\colon\!\!\!=}
\newcommand{\col}{\colon}

\DeclareMathOperator{\coker}{{coker}}

\DeclareMathOperator{\Ex}{Ex}
\DeclareMathOperator{\exc}{Exc}
\DeclareMathOperator{\excnklt}{Exc_{nklt}}

\newcommand{\sHom}[0]{{\mcH\mco\mcm}}

\DeclareMathOperator{\lc}{{lc}}

\DeclareMathOperator{\Ob}{{Ob}}
\DeclareMathOperator{\ob}{{Ob}}

\DeclareMathOperator{\red}{red}

\DeclareMathOperator{\Sing}{{Sing}}

\DeclareMathOperator{\supp}{{supp}}

\newcommand{\factor}[2]{\left. \raise 2pt\hbox{\ensuremath{#1}} \right/
        \hskip -2pt\raise -2pt\hbox{\ensuremath{#2}}}

\newcommand{\disjoint}{\overset{_\kdot}\cup}
\newcommand{\myR}{{\mcR\!}}

\newcommand{\tld}{\widetilde }
\newcommand{\blank}{\underline{\hskip 10pt}}

\newcommand{\kdot}{{{\,\begin{picture}(1,1)(-1,-2)\circle*{2}\end{picture}\ }}}
\newcommand{\mydot}{\kdot}
\newcommand{\cx}[1]{{#1}^{\raisebox{.15em}{\ensuremath\kdot}}}
\newcommand{\DuBois}[1]{{\underline \Omega {}^0_{#1}}}
\newcommand{\FullDuBois}[1]{{\underline \Omega {}^{\mydot}_{#1}}}
\newcommand{\Om}{\underline{\Omega}}

\def\coh#1.#2.#3.{H^{#1}(#2,#3)}
\def\dimcoh#1.#2.#3.{h^{#1}(#2,#3)}
\def\hypcoh#1.#2.#3.{\mathbb H_{\vphantom{l}}^{#1}(#2,#3)}
\def\loccoh#1.#2.#3.#4.{H^{#1}_{#2}(#3,#4)}
\def\dimloccoh#1.#2.#3.#4.{h^{#1}_{#2}(#3,#4)}
\def\lochypcoh#1.#2.#3.#4.{\mathbb H^{#1}_{#2}(#3,#4)}
\def\ses#1.#2.#3.{0  \longrightarrow  #1   \longrightarrow 
 #2 \longrightarrow #3 \longrightarrow 0} 
\def\sesshort#1.#2.#3.{0
 \rightarrow #1 \rightarrow #2 \rightarrow #3 \rightarrow 0}
\def\dist#1.#2.#3.{  #1   \longrightarrow 
 #2 \longrightarrow #3 \stackrel{+1}{\longrightarrow} } 
\def\CDdist#1.#2.#3.{  #1   @>>>  #2  @>>>   #3 @>+1>> }  
\def\shortses#1.#2.#3.{0  \rightarrow  #1   \rightarrow 
 #2  \rightarrow   #3 \rightarrow  0}
\def\shortdist#1.#2.#3.{  #1   \rightarrow 
 #2  \rightarrow   #3 \stackrel{+1}{\rightarrow} }  
\def\ddist#1.#2.#3.#4.#5.#6.{\CD
#1 @>>> #2 @>>> #3 @>+1>> \\
@VVV @VVV @VVV \\
#4 @>>> #5 @>>> #6 @>+1>> 
\endCD}
\def\ddistun#1.#2.#3.#4.#5.#6.{\CD
#1 @>>> #2 @>>> #3 @>+1>> \\
@. @VVV @VVV  \\
#4 @>>> #5 @>>> #6 @>+1>> 
\endCD}
\def\Iff#1#2#3{
\hfil\hbox{\hsize =#1
\vtop{\noin #2}
\hskip.5cm 
\lower.5\baselineskip\hbox{$\Leftrightarrow$}\hskip.5cm
\vtop{\noin #3}}\hfil\medskip}
\newcommand{\union}\cup
\newcommand{\intersect}\cap
\newcommand{\Union}\bigcup
\newcommand{\Intersect}\bigcap
\def\myoplus#1.#2.{\underset #1 \to {\overset #2 \to \oplus}}

\newcommand{\resto}{\big\vert_}

\def\qis{\,{\simeq}_{\text{qis}}\,}

\begin{document}
\makeatletter
\newenvironment{refmr}{}{}
\renewcommand{\labelenumi}{\hskip .5em(\thethm.\arabic{enumi})}
\renewcommand\thesubsection{\thesection.\Alph{subsection}}
\renewcommand\subsection{
  \renewcommand{\sfdefault}{pag}
  \@startsection{subsection}%
  {2}{0pt}{-\baselineskip}{.2\baselineskip}{\raggedright
    \sffamily\itshape\small
  }}
\renewcommand\section{
  \renewcommand{\sfdefault}{phv}
  \@startsection{section} %
  {1}{0pt}{\baselineskip}{.2\baselineskip}{\centering
    \sffamily
    \scshape
}}
\newcounter{lastyear}\setcounter{lastyear}{\the\year}
\addtocounter{lastyear}{-1}
\newcommand\sideremark[1]{%
\normalmarginpar
\marginpar
[
\hskip .45in
\begin{minipage}{.75in}
\tiny #1
\end{minipage}
]
{
\hskip -.075in
\begin{minipage}{.75in}
\tiny #1
\end{minipage}
}}
\newcommand\rsideremark[1]{
\reversemarginpar
\marginpar
[
\hskip .45in
\begin{minipage}{.75in}
\tiny #1
\end{minipage}
]
{
\hskip -.075in
\begin{minipage}{.75in}
\tiny #1
\end{minipage}
}}
\newcommand\Index[1]{{#1}\index{#1}}
\newcommand\inddef[1]{\emph{#1}\index{#1}}
\newcommand\noin{\noindent}
\newcommand\hugeskip{\bigskip\bigskip\bigskip}
\newcommand\smc{\sc}
\newcommand\dsize{\displaystyle}
\newcommand\sh{\subheading}
\newcommand\nl{\newline}
\newcommand\get{\input /home/kovacs/tex/latex/} 
\newcommand\Get{\Input /home/kovacs/tex/latex/} 
\newcommand\toappear{\rm (to appear)}
\newcommand\mycite[1]{[#1]}
\newcommand\myref[1]{(\ref{#1})}
\newcommand\myli{\hfill\newline\smallskip\noindent{$\bullet$}\quad}
\newcommand\vol[1]{{\bf #1}\ } 
\newcommand\yr[1]{\rm (#1)\ } 
\newcommand\cf{cf.\ \cite}
\newcommand\mycf{cf.\ \mycite}
\newcommand\te{there exist}
\newcommand\st{such that}
\newcommand\myskip{3pt}
\newtheoremstyle{bozont}{3pt}{3pt}%
     {\itshape}
     {}
     {\bfseries}
     {.}
     {.5em}
     {\thmname{#1}\thmnumber{ #2}\thmnote{ \rm #3}}
\newtheoremstyle{bozont-sf}{3pt}{3pt}%
     {\itshape}
     {}
     {\sffamily}
     {.}
     {.5em}
     {\thmname{#1}\thmnumber{ #2}\thmnote{ \rm #3}}
\newtheoremstyle{bozont-sc}{3pt}{3pt}%
     {\itshape}
     {}
     {\scshape}
     {.}
     {.5em}
     {\thmname{#1}\thmnumber{ #2}\thmnote{ \rm #3}}
\newtheoremstyle{bozont-remark}{3pt}{3pt}%
     {}
     {}
     {\scshape}
     {.}
     {.5em}
     {\thmname{#1}\thmnumber{ #2}\thmnote{ \rm #3}}
\newtheoremstyle{bozont-def}{3pt}{3pt}%
     {}
     {}
     {\bfseries}
     {.}
     {.5em}
     {\thmname{#1}\thmnumber{ #2}\thmnote{ \rm #3}}
\newtheoremstyle{bozont-reverse}{3pt}{3pt}%
     {\itshape}
     {}
     {\bfseries}
     {.}
     {.5em}
     {\thmnumber{#2.}\thmname{ #1}\thmnote{ \rm #3}}
\newtheoremstyle{bozont-reverse-sc}{3pt}{3pt}%
     {\itshape}
     {}
     {\scshape}
     {.}
     {.5em}
     {\thmnumber{#2.}\thmname{ #1}\thmnote{ \rm #3}}
\newtheoremstyle{bozont-reverse-sf}{3pt}{3pt}%
     {\itshape}
     {}
     {\sffamily}
     {.}
     {.5em}
     {\thmnumber{#2.}\thmname{ #1}\thmnote{ \rm #3}}
\newtheoremstyle{bozont-remark-reverse}{3pt}{3pt}%
     {}
     {}
     {\sc}
     {.}
     {.5em}
     {\thmnumber{#2.}\thmname{ #1}\thmnote{ \rm #3}}
\newtheoremstyle{bozont-def-reverse}{3pt}{3pt}%
     {}
     {}
     {\bfseries}
     {.}
     {.5em}
     {\thmnumber{#2.}\thmname{ #1}\thmnote{ \rm #3}}
\newtheoremstyle{bozont-def-newnum-reverse}{3pt}{3pt}%
     {}
     {}
     {\bfseries}
     {}
     {.5em}
     {\thmnumber{#2.}\thmname{ #1}\thmnote{ \rm #3}}
\theoremstyle{bozont}    
\ifnum \value{are-there-sections}=0 {%
  \newtheorem{proclaim}{Theorem}
} 
\else {%
  \newtheorem{proclaim}{Theorem}[section]
} 
\fi
\newtheorem{thm}[proclaim]{Theorem}
\newtheorem{mainthm}[proclaim]{Main Theorem}
\newtheorem{cor}[proclaim]{Corollary} 
\newtheorem{cors}[proclaim]{Corollaries} 
\newtheorem{lem}[proclaim]{Lemma} 
\newtheorem{prop}[proclaim]{Proposition} 
\newtheorem{conj}[proclaim]{Conjecture}
\newtheorem{subproclaim}[equation]{Theorem}
\newtheorem{subthm}[equation]{Theorem}
\newtheorem{subcor}[equation]{Corollary} 
\newtheorem{sublem}[equation]{Lemma} 
\newtheorem{subprop}[equation]{Proposition} 
\newtheorem{subconj}[equation]{Conjecture}
\theoremstyle{bozont-sc}
\newtheorem{proclaim-special}[proclaim]{\specialthmname}
\newenvironment{proclaimspecial}[1]
     {\def\specialthmname{#1}\begin{proclaim-special}}
     {\end{proclaim-special}}
\theoremstyle{bozont-remark}
\newtheorem{rem}[proclaim]{Remark}
\newtheorem{subrem}[equation]{Remark}
\newtheorem{notation}[proclaim]{Notation} 
\newtheorem{assume}[proclaim]{Assumptions} 
\newtheorem{obs}[proclaim]{Observation} 
\newtheorem{example}[proclaim]{Example} 
\newtheorem{examples}[proclaim]{Examples} 
\newtheorem{complem}[equation]{Complement}
\newtheorem{const}[proclaim]{Construction}   
\newtheorem{ex}[proclaim]{Exercise} 
\newtheorem{subnotation}[equation]{Notation} 
\newtheorem{subassume}[equation]{Assumptions} 
\newtheorem{subobs}[equation]{Observation} 
\newtheorem{subexample}[equation]{Example} 
\newtheorem{subex}[equation]{Exercise} 
\newtheorem{claim}[proclaim]{Claim} 
\newtheorem{inclaim}[equation]{Claim} 
\newtheorem{subclaim}[equation]{Claim} 
\newtheorem{case}{Case} 
\newtheorem{subcase}{Subcase}   
\newtheorem{step}{Step}
\newtheorem{approach}{Approach}
\newtheorem{fact}{Fact}
\newtheorem{subsay}{}
\newtheorem*{SubHeading*}{\SubHeadingName}%
\newtheorem{SubHeading}[proclaim]{\SubHeadingName}
\newtheorem{sSubHeading}[equation]{\sSubHeadingName}
\newenvironment{demo}[1] {\def\SubHeadingName{#1}\begin{SubHeading}}
  {\end{SubHeading}}%
\newenvironment{subdemo}[1]{\def\sSubHeadingName{#1}\begin{sSubHeading}}
  {\end{sSubHeading}} %
\newenvironment{demo-r}[1]{\def\SubHeadingName{#1}\begin{SubHeading-r}}
  {\end{SubHeading-r}}%
\newenvironment{subdemo-r}[1]{\def\sSubHeadingName{#1}\begin{sSubHeading-r}}
  {\end{sSubHeading-r}} %
\newenvironment{demo*}[1]{\def\SubHeadingName{#1}\begin{SubHeading*}}
  {\end{SubHeading*}}%
\newtheorem{defini}[proclaim]{Definition}
\newtheorem{question}[proclaim]{Question}
\newtheorem{subquestion}[equation]{Question}
\newtheorem{crit}[proclaim]{Criterion}
\newtheorem{pitfall}[proclaim]{Pitfall}
\newtheorem{addition}[proclaim]{Addition}
\newtheorem{principle}[proclaim]{Principle} 
\newtheorem{condition}[proclaim]{Condition}
\newtheorem{say}[proclaim]{}
\newtheorem{exmp}[proclaim]{Example}
\newtheorem{hint}[proclaim]{Hint}
\newtheorem{exrc}[proclaim]{Exercise}
\newtheorem{prob}[proclaim]{Problem}
\newtheorem{ques}[proclaim]{Question}    
\newtheorem{alg}[proclaim]{Algorithm}
\newtheorem{remk}[proclaim]{Remark}          
\newtheorem{note}[proclaim]{Note}            
\newtheorem{summ}[proclaim]{Summary}         
\newtheorem{notationk}[proclaim]{Notation}   
\newtheorem{warning}[proclaim]{Warning}  
\newtheorem{defn-thm}[proclaim]{Definition--Theorem}  
\newtheorem{convention}[proclaim]{Convention}  
\newtheorem*{ack}{Acknowledgment}
\newtheorem*{acks}{Acknowledgments}
\theoremstyle{bozont-def}    
\newtheorem{defn}[proclaim]{Definition}
\newtheorem{subdefn}[equation]{Definition}
\theoremstyle{bozont-reverse}    
\newtheorem{corr}[proclaim]{Corollary} 
\newtheorem{lemr}[proclaim]{Lemma} 
\newtheorem{propr}[proclaim]{Proposition} 
\newtheorem{conjr}[proclaim]{Conjecture}
\theoremstyle{bozont-reverse-sc}
\newtheorem{proclaimr-special}[proclaim]{\specialthmname}
\newenvironment{proclaimspecialr}[1]%
{\def\specialthmname{#1}\begin{proclaimr-special}}%
{\end{proclaimr-special}}
\theoremstyle{bozont-remark-reverse}
\newtheorem{remr}[proclaim]{Remark}
\newtheorem{subremr}[equation]{Remark}
\newtheorem{notationr}[proclaim]{Notation} 
\newtheorem{assumer}[proclaim]{Assumptions} 
\newtheorem{obsr}[proclaim]{Observation} 
\newtheorem{exampler}[proclaim]{Example} 
\newtheorem{exr}[proclaim]{Exercise} 
\newtheorem{claimr}[proclaim]{Claim} 
\newtheorem{inclaimr}[equation]{Claim} 
\newtheorem{SubHeading-r}[proclaim]{\SubHeadingName}
\newtheorem{sSubHeading-r}[equation]{\sSubHeadingName}
\newtheorem{SubHeadingr}[proclaim]{\SubHeadingName}
\newtheorem{sSubHeadingr}[equation]{\sSubHeadingName}
\newenvironment{demor}[1]{\def\SubHeadingName{#1}\begin{SubHeadingr}}{\end{SubHeadingr}}
\newtheorem{definir}[proclaim]{Definition}
\theoremstyle{bozont-def-newnum-reverse}    
\newtheorem{newnumr}[proclaim]{}
\theoremstyle{bozont-def-reverse}    
\newtheorem{defnr}[proclaim]{Definition}
\newtheorem{questionr}[proclaim]{Question}
\newtheorem{newnumspecial}[proclaim]{\specialnewnumname}
\newenvironment{newnum}[1]{\def\specialnewnumname{#1}\begin{newnumspecial}}{\end{newnumspecial}}
\numberwithin{equation}{proclaim}
\numberwithin{figure}{section} 
\newcommand\equinsect{\numberwithin{equation}{section}}
\newcommand\equinthm{\numberwithin{equation}{proclaim}}
\newcommand\figinthm{\numberwithin{figure}{proclaim}}
\newcommand\figinsect{\numberwithin{figure}{section}}
\newenvironment{sequation}{%
\numberwithin{equation}{section}%
\begin{equation}%
}{%
\end{equation}%
\numberwithin{equation}{thm}%
\addtocounter{thm}{1}%
}
\newcommand{\num}{\arabic{section}.\arabic{proclaim}}
\newenvironment{pf}{\smallskip \noindent {\sc Proof. }}{\qed\smallskip}
\newenvironment{enumerate-p}{
  \begin{enumerate}}
  {\setcounter{equation}{\value{enumi}}\end{enumerate}}
\newenvironment{enumerate-cont}{
  \begin{enumerate}
    {\setcounter{enumi}{\value{equation}}}}
  {\setcounter{equation}{\value{enumi}}
  \end{enumerate}}
\let\lenumi\labelenumi
\newcommand{\rmlabels}{\renewcommand{\labelenumi}{\rm \lenumi}}
\newcommand{\rmlabelsoff}{\renewcommand{\labelenumi}{\lenumi}}
\newenvironment{heading}{\begin{center} \sc}{\end{center}}
\newcommand\subheading[1]{\smallskip\noindent{{\bf #1.}\ }}
\newlength{\swidth}
\setlength{\swidth}{\textwidth}
\addtolength{\swidth}{-,5\parindent}
\newenvironment{narrow}{
  \medskip\noindent\hfill\begin{minipage}{\swidth}}
  {\end{minipage}\medskip}
\newcommand\nospace{\hskip-.45ex}
\makeatother

\title{DB pairs and vanishing theorems}
\author{\me}
\date{\today}
\thanks{\mythanks}
\address{\myaddress}
\email{\myemail}
\urladdr{\myurladdr}
%
\subjclass[2010]{14J17}
\maketitle
\newcommand{\szabores}{Szab\'o-resolution\xspace}
\newcommand{\pairD}{\Delta}

\centerline{\sf In memoriam Professor Masayoshi Nagata}

\begin{abstract}
  The main purpose of this article is to define the notion of \uj{Du~Bois} singularities
  for pairs and proving a vanishing theorem using this new notion. The main vanishing
  theorem specializes to a new vanishing theorem for resolutions of log canonial
  singularities.
\end{abstract}

\section{Introduction}

The class of rational singularities is one of the most important classes of
singularities.  Their essence lies in the fact that their cohomological behavior is
very similar to that of smooth points. For instance, vanishing theorems can be easily
extended to varieties with rational singularities.  Establishing that a certain class
of singularities is rational opens the door to using very powerful tools on varieties
with those singularities.

Du~Bois (or DB) singularities are probably somewhat harder to appreciate at first,
but they are equally important. Their main importance comes from two facts: They are
not too far from rational singularities, that is, they share many of their
properties, but the class of \uj{Du~Bois} singularities is more inclusive than that of
rational singularities.  For instance, log canonical singularities are \uj{Du~Bois}, but
not necessarily rational. The class of \uj{Du~Bois} singularities is also more stable
under degeneration.

Recently there has been an effort to extend the notion of rational singularities to
pairs. There are at least two approaches; Schwede and Takagi \cite{MR2492473} are
dealing with pairs $(X,\Delta)$ where $\rdown{\Delta}=0$ while Koll\'ar and Kov\'acs
\cite{KollarKovacsRP} are studying pairs $(X,\Delta)$ where $\Delta$ is reduced.

The main goal of this article is to extend the definition of \uj{Du~Bois} singularities to
pairs in the spirit of the latter approach.

Here is a brief overview of the paper.

In section~\ref{sec:rational-du-bois} some basic properties of rational and DB
singularities are reviewed, a few new ones are introduced, and the DB defect is
defined. In section~\ref{sec:pairs-gener-pairs} I recall the definition and some
basic properties of pairs, generalized pairs, and rational pairs. I define the notion
of a DB pair and the DB defect of a generalized pair and prove a few basic
properties.  In section~\ref{sec:cohom-with-comp} I recall a relevant theorem from
Deligne's Hodge theory and derive a corollary that will be needed later. In
section~\ref{sec:db-pairs} one of the main results is proven. A somewhat weaker
version is the following. See Theorem~\ref{thm:wdb-is-db} for the stronger statement.

\begin{thm}
  Rational pairs are DB pairs.
\end{thm}

This generalizes \cite[Theorem~S]{Kovacs99} and \cite[5.4]{MR1741272} to pairs.  In
section \ref{sec:vanishing-theorems} I prove a rather general vanishing theorem for
DB pairs and use it to derive the following vanishing theorem for log canonical
pairs.

\begin{thm}\label{thm:main}
  Let $(X, \pairD)$ be a \uj{$\bQ$-factorial} log canonical pair, $\pi : \widetilde X
  \to X$ a log resolution of $(X,
  \pairD)$. 
  Let $\wt \pairD=\bigl(\pi^{-1}_*\rdown \pairD + \excnklt(\pi)\bigr)_{\red}$.  Then
  \begin{equation*}
    \myR^{i}\pi_* \, \sO_{\widetilde X}(- \widetilde \pairD) = 0\quad\text{for $i>0$.}\\
  \end{equation*}
\end{thm}

A philosophical consequence one might draw from this theorem is that log canonical
pairs are not too far from being rational. One may even view this a vanishing theorem
similar to the one in the definition of rational singularities cf.\
\eqref{def:rtl-sing}, \eqref{def:rat-sing-pairs} with a correction term as in
vanishing theorems with multiplier ideals. Notice however, that this is in a dual
form compared to Nadel's vanishing, and hence does not follow from that, especially
since the target is not necessarily Cohen-Macaulay.

Theorem~\ref{thm:main} is also closely related to Steenbrink's characterization of
normal isolated \uj{Du~Bois} singularities \cite[3.6]{SteenbrinkMixed} (cf.\
\cite[4.13]{DuBois81}, \cite[6.1]{Kovacs-Schwede09}).

A weaker version of this theorem was the corner stone of a recent result on extending
differential forms to a log resolution \cite{GKKP10}. For details on how this theorem
may be applied, see the original article. It is possible that the current theorem
will lead to a strengthening of that result.

\begin{demo}{\bf Definitions and Notation}\label{demo:defs-and-not}
  Unless otherwise stated, all objects are assumed to be defined over $\bC$, all
  schemes are assumed to be of finite type over $\bC$ and a morphism means a morphism
  between schemes of finite type over $\bC$.

  If $\phi:Y\to Z$ is a birational morphism, then $\exc(\phi)$ will denote the
  \emph{exceptional set} of $\phi$. For a closed subscheme $W\subseteq X$, the ideal
  sheaf of $W$ is denoted by $\sI_{W\subseteq X}$ or if no confusion is likely, then
  simply by $\sI_W$.  For a point $x\in X$, $\kappa(x)$ denotes the residue field of
  $\sO_{X,x}$.
  
  For morphisms $\phi:X\to B$ and $\vartheta: T\to B$, the symbol $X_T$ will denote
  $X\times_B T$ and $\phi_T:X_T\to T$ the induced morphism.  In particular, for $b\in
  B$ we write $X_b = \phi^{-1}(b)$.
  Of course, by symmetry, we also have the notation $\vartheta_X:T_X\simeq X_T\to X$
  and if $\sF$ is an $\sO_X$-module, then $\sF_T$ will denote the $\sO_{X_T}$-module
  $\vartheta_X^*\sF$.




  Let $X$ be a complex scheme (i.e., a scheme of finite type over $\bC$) of dimension
  n. Let $D_{\rm filt}(X)$ denote the derived category of filtered complexes of
  $\sO_{X}$-modules with differentials of order $\leq 1$ and $D_{\rm filt, coh}(X)$
  the subcategory of $D_{\rm filt}(X)$ of complexes $\cx K$, such that for all $i$,
  the cohomology sheaves of $Gr^{i}_{\rm filt}K^{\kdot}$ are coherent cf.\
  \cite{DuBois81}, \cite{GNPP88}.  Let $D(X)$ and $D_{\rm coh}(X)$ denote the derived
  categories with the same definition except that the complexes are assumed to have
  the trivial filtration.  The superscripts $+, -, b$ carry the usual meaning
  (bounded below, bounded above, bounded).  Isomorphism in these categories is
  denoted by $\qis$.  A sheaf $\sF$ is also considered as a complex $\sF^\kdot$ with
  $\sF^0=\sF$ and $\sF^i=0$ for $i\neq 0$.  If $K^{\kdot}$ is a complex in any of the
  above categories, then $h^i(K^{\kdot})$ denotes the $i$-th cohomology sheaf of
  $K^{\kdot}$.

  The right derived functor of an additive functor $F$, if it exists, is denoted by
  $\myR F$ and $\myR^iF$ is short for $h^i\circ \myR F$. Furthermore, $\bH^i$,
  $\bH^i_{\rm c}$, $\bH^i_Z$ , and $\sH^i_Z$ will denote $\myR^i\Gamma$,
  $\myR^i\Gamma_{\rm c}$, $\myR^i\Gamma_Z$, and $\myR^i\sH_Z$ respectively, where
  $\Gamma$ is the functor of global sections, $\Gamma_{\rm c}$ is the functor of
  global sections with proper support, $\Gamma_Z$ is the functor of global sections
  with support in the closed subset $Z$, and $\sH_Z$ is the functor of the sheaf of
  local sections with support in the closed subset $Z$.  Note that according to this
  terminology, if $\phi\col Y\to X$ is a morphism and $\sF$ is a coherent sheaf on
  $Y$, then $\myR\phi_*\sF$ is the complex whose cohomology sheaves give rise to the
  usual higher direct images of $\sF$.

  We will often use the notion that a morphism ${f}: \sfA\to \sfB$ in a derived
  category \emph{has a left inverse}. This means that there exists a morphism
  $f^\ell: \sfB\to \sfA$ in the same derived category such that
  $f^\ell\circ{f}:\sfA\to\sfA$ is the identity morphism of $\sfA$. I.e., $f^\ell$ is
  a \emph{left inverse} of ${f}$.

  Finally, we will also make the following simplification in notation. First observe
  that if $\iota:\Sigma \into X$ is a closed embedding of schemes then $\iota_*$ is
  exact and hence $\myR\iota_*=\iota_*$. This allows one to make the following
  harmless abuse of notation: If $\sfA\in\ob D(\Sigma)$, then, as usual for sheaves,
  we will drop $\iota_*$ from the notation of the object $\iota_*\sfA$. In other
  words, we will, without further warning, consider $\sfA$ an object in $D(X)$.
\end{demo}

\begin{ack}
  I would like to thank Donu Arapura for explaining some of the intricacies of the
  relevant Hodge theory to me\uj{, Osamu Fujino for helpful remarks,} and the referee
  for useful comments.
\end{ack}

\section{Rational and {Du~Bois} singularities}\label{sec:rational-du-bois}

\begin{defn}\label{def:rtl-sing}
  Let $X$ be a normal variety and $\phi :Y \rightarrow X$ a resolution of
  singularities. $X$ is said to have \emph{rational} singularities if
  $\myR^i\phi_*\sO_Y=0$ for all $i>0$, or equivalently if the natural map $\sO_X\to
  \myR\phi_*\sO_Y$ is a quasi-isomorphism.
\end{defn}

\uj{Du~Bois} singularities are defined via Deligne's Hodge theory 
We will need a little preparation before we can define them.

The starting point is \uj{Du~Bois}'s construction, following Deligne's ideas, of the
generalized de~Rham complex, which we call the \emph{\uj{Deligne-Du~Bois} complex}.
Recall, that if $X$ is a smooth complex algebraic variety of dimension $n$, then the
sheaves of differential $p$-forms with the usual exterior differentiation give a
resolution of the constant sheaf $\bC_X$. I.e., one has a filtered complex of
sheaves,
$$
\xymatrix{%
\sO_X \ar[r]^{d} & \Omega_X^1 \ar[r]^{d} & \Omega_X^2 \ar[r]^{d} & \Omega_X^3
\ar[r]^{d} & \dots  \ar[r]^{d} & \Omega_X^n\simeq \omega_X,
}
$$
which is quasi-isomorphic to the constant sheaf $\bC_X$ via the natural map $\bC_X\to
\sO_X$ given by considering constants as holomorphic functions on $X$. Recall that
this complex \emph{is not} a complex of quasi-coherent sheaves. The sheaves in the
complex are quasi-coherent, but the maps between them are not $\sO_X$-module
morphisms. Notice however that this is actually not a shortcoming; as $\bC_X$ is not
a quasi-coherent sheaf, one cannot expect a resolution of it in the category of
quasi-coherent sheaves.

The \uj{Deligne-Du~Bois} complex is a generalization of the de~Rham complex to singular
varieties.  It is a complex of sheaves on $X$ that is quasi-isomorphic to the
constant sheaf $\bC_X$. The terms of this complex are harder to describe but its
properties, especially cohomological properties are very similar to the de~Rham
complex of smooth varieties. In fact, for a smooth variety the \uj{Deligne-Du~Bois}
complex is quasi-isomorphic to the de~Rham complex, so it is indeed a direct
generalization.

The construction of this complex, $\FullDuBois{X}$, is based on simplicial
resolutions. The reader interested in the details is referred to the original article
\cite{DuBois81}.  Note also that a simplified construction was later obtained in
\cite{Carlson85} and \cite{GNPP88} via the general theory of polyhedral and cubic
resolutions.  An easily accessible introduction can be found in \cite{Steenbrink85}.
Other useful references are the recent book \cite{PetersSteenbrinkBook} and the
survey \cite{Kovacs-Schwede09}. We will actually not use these resolutions here. They
are needed for the construction, but if one is willing to believe the listed
properties (which follow in a rather straightforward way from the construction) then
one should be able follow the material presented here.  The interested reader should
note that recently Schwede found a simpler alternative construction of (part of) the
\uj{Deligne-Du~Bois} complex that does not need a simplicial resolution \cite{MR2339829}.
For applications of the \uj{Deligne-Du~Bois} complex and \uj{Du~Bois} singularities other than
the ones listed here see \cite{SteenbrinkMixed}, \cite[Chapter 12]{Kollar95s},
\cite{Kovacs99,Kovacs00c,KSS10,KK10}.

The word ``hyperresolution'' will refer to either a simplicial, polyhedral, or cubic
resolution. Formally, the construction of $\FullDuBois{X}$ is the same regardless the
type of resolution used and no specific aspects of either types will be used.

The next theorem lists the basic properties of the \uj{Deligne-Du~Bois} complex:

\begin{thm}[{\cite{DuBois81}}]\label{defDB}
  Let $X$ be a complex scheme of finite type
  . Then there exists a functorially defined object $\Om_X^{\kdot} \in \Ob D_{\rm
    filt}(X)$ such that using the notation
  $$
  \Om_X^ p \leteq Gr^{p}_{\rm filt}\, \Om_X^\kdot [p],
  $$
  it satisfies the following properties
  \begin{enumerate-p}
  \item 
    $$
    \Om_X^\kdot \qis \bC_{X}.
    $$
    
  \item $\underline{\Omega}_{(\_)}^{\kdot}$ is functorial, i.e., if $\phi \col Y\to
    X$ is a morphism of complex schemes of finite type, then there exists a natural
    map $\phi^{*}$ of filtered complexes
    $$
    \phi^{*}\col \Om_X^\kdot \to \myR\phi_{*}\underline{\Omega}_Y^{\kdot}
    $$
    Furthermore, $\Om_X^\kdot  \in \Ob \left(D^{b}_{\rm filt, coh}(X)\right)$ and if
    $\phi$ is proper, then $\phi^{*}$ is a morphism in $D^{b}_{\rm filt, coh}(X)$.
    \label{functorial}
  
  \item Let $U \subseteq X$ be an open subscheme of $X$. Then
    $$
    \Om_X^\kdot \resto U \qis\Om^{\,\kdot}_U.
    $$
    
  \item If $X$ is proper, then there exists a spectral sequence degenerating at $E\uj{_1}$
    and abutting to the singular cohomology of $X$:
    $$
    E\uj{_1}^{pq}={\bH}^q \left(X, \Om_X^ p \right) \Rightarrow H^{p+q}(X, \bC).
    $$\label{item:Hodge}
  
  \item If\/ $\varepsilon_\kdot\col X_\kdot\to X$ is a hyperresolution, then
    $$
    \Om_X^\kdot \qis \myR{\varepsilon_\kdot}_* \Omega^\kdot_{X_\kdot}.
    $$
    In particular, $h^i\left(\Om_X^ p \right)=0$ for $i<0$.\label{item:1}
    
  \item There exists a natural map, $\sO_{X}\to \Om_X^ 0$, compatible with
    (\ref{defDB}.\ref{functorial}).
    \label{item:dR-to-DB}

  \item If\/ $X$ is a normal crossing divisor in a smooth variety, then
    $$
    \Om_X^\kdot \qis\Omega^\kdot_X.
    $$
    In particular,
    $$
    \Om_X^ p \qis\Omega^p_X.
    $$
    \label{item:8}
    
  \item If\/ $\phi\col Y\to X$ is a resolution of singularities, then
    $$
    \Om_X^{\dim X} \qis \myR\phi_*\omega_Y.
    $$
  \item\label{item:exact-triangle} Let $\pi : \tld X \rightarrow X$ be a projective
    morphism and $\Sigma \subseteq X$ a reduced closed subscheme such that $\pi$ is an
    isomorphism outside of $\Sigma$.  Let $E$ denote the reduced subscheme of $\tld
    X$ with support equal to $\pi^{-1}(X)$. 
    Then for each $p$ one has an exact triangle of objects in the derived category,
    $$
    \xymatrix{ \Om^p_X \ar[r] & \Om^p_\Sigma \oplus \myR \pi_* \Om^p_{\tld X}
      \ar[r]^-{-}
      & \myR      \pi_* \Om^p_E \ar[r]^-{+1} & .\\
    }
    $$
  \item \label{item:2} Suppose $X=Y\cup Z$ is the union of two closed subschemes and
    denote their intersection by $W\leteq Y\cap Z$.  Then for each $p$ one has an
    exact triangle of objects in the derived category,
    $$
    \xymatrix{ \Om^p_X \ar[r] & \Om^p_{Y} \oplus \Om^p_{Z}
      \ar[r]^-{-}  &  \Om^p_W \ar[r]^-{+1} & .\\
    }
    $$
  \end{enumerate-p}
\end{thm}

It turns out that the \uj{Deligne-Du~Bois} complex behaves very much like the de~Rham
complex for smooth varieties. Observe that (\ref{defDB}.\ref{item:Hodge}) says that
the Hodge-to-de~Rham (a.k.a.\ Fr\"olicher) spectral sequence works for singular
varieties if one uses the \uj{Deligne-Du~Bois} complex in place of the de~Rham complex.
This has far reaching consequences and if the associated graded pieces $\Om_X^ p$
turn out to be computable, then this single property leads to many applications.

Notice that (\ref{defDB}.\ref{item:dR-to-DB}) gives a natural map $\sO_{X}\to
\Om^0_X$, and we will be interested in situations when this map is a
quasi-isomorphism.  When $X$ is proper over $\bC$, such a quasi-isomorphism implies
that the natural map
\begin{equation*}
  H^i(X^{\rm an}, \bC) \rightarrow H^i(X, \sO_{X}) = \bH^i(X, \DuBois{X})
\end{equation*}
is surjective because of the degeneration at $E_1$ of the spectral sequence in
(\ref{defDB}.\ref{item:Hodge}). Notice that this is the condition that is crucial for
Kodaira-type vanishing theorems cf.\ \cite[\S 9]{Kollar95s}.

Following \uj{Du~Bois}, Steenbrink was the first to study this condition and he christened
this property after \uj{Du~Bois}. It should be noted that many of the ideas that play
important roles in this theory originated from Deligne. Unfortunately the now
standard terminology does not reflect this.

\begin{defn}\label{def:db-sing}
  A scheme $X$ is said to have \emph{\uj{Du~Bois}} singularities (or \emph{DB}
  singularities for short) if the natural map $\sO_{X}\to \Om^0_X$ from
  (\ref{defDB}.\ref{item:dR-to-DB}) is a quasi-isomorphism.
\end{defn}

\begin{rem}
  If $\varepsilon : X_{\kdot} \rightarrow X$ is a hyperresolution of $X$ then $X$ has
  DB singularities if and only if the natural map $\sO_X \rightarrow \myR
  {\varepsilon_{\kdot}}_* \sO_{X_{\kdot}}$ is a quasi-isomorphism.
\end{rem}

\begin{example}
  It is easy to see that smooth points are DB and Deligne proved that normal crossing
  singularities are DB as well cf.\ (\ref{defDB}.\ref{item:8}), \cite[Lemme
  2(b)]{MR0376678}.
\end{example}

In applications it is very useful to be able to take general hyperplane sections. The
next statement helps with that.

\begin{prop}\label{prop:db-cx-of-hyper}
  Let $X$ be a quasi-projective variety and $H\subset X$ a general member of a very
  ample linear system. Then $\Om_H^\kdot\qis \Om_X^\kdot\otimes_L\sO_H$. 
\end{prop}

\begin{proof}
  Let $\varepsilon_\kdot: X_\kdot\to X$ be a hyperresolution. Since $H$ is general,
  the fiber product $X_\kdot\times_XH\to H$ provides a hyperresolution of $H$. Then
  the statement follows from (\ref{defDB}.\ref{item:1}) applied to both $X$ and $H$.
\end{proof}

We saw in (\ref{defDB}.\ref{item:1}) that $h^i\left(\Om_X^0 \right)=0$ for $i<0$.  In
fact, there is a corresponding upper bound by \cite[III.1.17]{GNPP88}, namely that
$h^i\left(\Om_X^0 \right)=0$ for $i>\dim X$. It turns out that one can make a
slightly better estimate.

\begin{prop}[\protect{cf.\ \cite[13.7]{GKKP10}, \cite[4.9]{KSS10}}]
  \label{prop:top-coh-vanishes}
  Let $X$ be a positive dimensional variety (i.e., reduced). Then the $i^{\text{th}}$
  cohomology sheaf of $\ul{\Omega}_X^p$ vanishes for all $i\geq \dim X$, i.e.,
  $h^i(\ul{\Omega}_X^p)=0$ for all $p$ and for all $i\geq \dim X$.
\end{prop}

\begin{proof} 
  For $i>\dim X$ or $p>0$, the statement follows from \cite[III.1.17]{GNPP88}.  The
  case $p=0$ and $i=n:=\dim X$ follows from either \cite[13.7]{GKKP10} or
  \cite[4.9]{KSS10}.
\end{proof}

\noindent
Another, much simpler fact that will be used later is the following:
\begin{cor}
  If $\dim X=1$, then $h^i(\Om_X^p)=0$ for $i\neq 0$. In particular $X$ is DB if and
  only if it is semi-normal.
\end{cor}

\begin{proof}
  The first statement is a direct consequence of \eqref{prop:top-coh-vanishes}.  For
  the last statement recall that the seminormalization of $\sO_X$ is exactly
  $h^0(\Om^0_X)$, and so $X$ is seminormal if and only if $\sO_X\simeq h^0(\Om^0_X)$
  \cite[5.2]{MR1741272} (cf.\ \cite[5.4.17]{Schwede06}, \cite[4.8]{MR2339829}, and
  \cite[5.6]{MR2503989}).  
\end{proof}

\begin{defn}\label{def:db-defect}
  The \emph{DB defect} of $X$ is the mapping cone of the morphism $\sO_X\to \Om^0_X$.
  It is denoted by $\Om_X^\times$. As a simple consequence of the definition, one has
  an exact triangle,
  $$
  \xymatrix{%
    \sO_X \ar[r] & \Om^0_X \ar[r] & \Om^\times_X \ar[r]^-{+1} & .
  }
  $$
  Notice that $h^0(\Om_X^\times)\simeq h^0(\Om_X^0)/\sO_X$ and
  $h^i(\Om_X^\times)\simeq h^i(\Om_X^0)$ for $i>0$.
\end{defn}

\begin{prop}\label{prop:hyper-plane-cuts}
  Let $X$ be a quasi-projective variety and $H\subset X$ a general member of a very
  ample linear system. Then $\Om_H^\times\qis \Om_X^\times\otimes_L\sO_H$. 
\end{prop}

\begin{proof}
  This follows easily from the definition and \ref{prop:db-cx-of-hyper}.
\end{proof}

The next simple observation explains the name of the DB defect.

\begin{lem}\label{lem:db-defect}
  A variety $X$ is DB if and only if the DB defect of $X$ is acyclic, that is,
  $\Om_X^\times\qis 0$.
\end{lem}

\begin{proof}
  This follows directly from the definition.
\end{proof}

\begin{prop}\label{prop:DB-defect-for-unions}
  Let $X=Y\cup Z$ be a union of closed subschemes with intersection $W=Y\cap Z$. Then
  one has an exact triangle of the DB defects of $X,Y,Z$, and $W$:
  \begin{equation*}
    \xymatrix{ \Om^\times_X \ar[r] & \Om^\times_{Y} \oplus \Om^\times_{Z}
     \ar[r]^-{-}  &  \Om^\times_W \ar[r]^-{+1} & .\\
    }
  \end{equation*}
\end{prop}

\begin{proof}
  Recall that there is an analogous exact triangle (a.k.a.\ a short exact sequence)
  for the structure sheaves of $X,Y,Z$, and $W$, which forms a commutative diagram
  with the exact triangle of (\ref{defDB}.\ref{item:2}),
  $$
  \xymatrix{ 
    \sO_X \ar[d] \ar[r] & \sO_{Y} \oplus \sO_{Z}
    \ar[d] \ar[r]^-{-}  &  \sO_W \ar[d] \ar[r]^-{+1} & \\
    \Om^0_X \ar[r] & \Om^0_{Y} \oplus \Om^0_{Z}
    \ar[r]^-{-}  &  \Om^0_W \ar[r]^-{+1} & .\\
  }
  $$
  Then the statement follows by the (derived category version of the) 9-lemma.
\end{proof}

\section{Pairs and generalized pairs}\label{sec:pairs-gener-pairs}

\subsection{Basic definitions}\label{ssec:basic-definitions}

For an arbitrary proper birational morphism, $\phi:Y\to X$, $\exc(\phi)$ stands for
the exceptional locus of $\phi$.  A \emph{$\bQ$-divisor} is a $\bQ$-linear
combination of integral Weil divisors; $\pairD=\sum a_i\pairD_i$, $a_i\in\bQ$,
$\pairD_i$ (integral) Weil divisor. For a $\bQ$-divisor $\pairD$, its
\emph{round-down} is defined by the formula: $\rdown \pairD=\sum
\rdown{a_i}\pairD_i$, where $\rdown{a_i}$ is the largest integer not larger than
$a_i$.

  A \emph{log variety} or \emph{pair} $(X,\pairD)$ consists of an equidimensional
  variety (i.e., a reduced scheme of finite type over a field $k$) $X$ and an
  effective $\bQ$-divisor $\pairD\subseteq X$. 
  A morphism of pairs $\phi:(Y,B)\to (X,\pairD)$ is a morphism $\phi:Y\to X$ such
  that $\phi(\supp B)\subseteq \supp\pairD$.

Let $(X,\pairD)$ be a pair with $\pairD$ a reduced integral divisor. Then
$(X,\pairD)$ is said to have \emph{simple normal crossings} or to be an \emph{snc
  pair at $p\in X$} if $X$ is smooth at $p$, and there are local coordinates
$x_1,\dots,x_n$ on $X$ in a neighbourhood of $p$ such that $\supp\pairD\subseteq
(x_1\cdots x_n=0)$ near $p$.  $(X,\pairD)$ is \emph{snc} if it is snc at every $p\in
X$.

  A morphism of pairs $\phi:(Y,\pairD_Y)\to (X,\pairD)$ is a \emph{log resolution of
    $(X,\pairD)$} if $\phi:Y\to X$ is proper and birational,
  $\pairD_Y=\phi^{-1}_*\pairD$, and $(\pairD_Y)_{\red}+\exc(\phi)$ is an snc divisor
  on $Y$.

  Note that we allow $(X,\pairD)$ to be snc and still call a morphism with these
  properties a log resolution. Also note that the notion of a log resolution is not
  used consistently in the literature.

  If $(X,\Delta)$ is a pair, then $\Delta$ is called a \emph{boundary} if
  $\rdown{(1-\varepsilon)\Delta}=0$ for all $0<\varepsilon<1$, i.e., the coefficients
  of all irreducible components of $\Delta$ are in the interval $[0,1]$.  For the
  definition of \emph{klt, dlt}, and \emph{lc} pairs see \cite{KM98}.
  %
  Let $(X, \Delta)$ be a pair and $\mu:X^{\rm m}\to X$ a proper birational
  \uj{morphism.}
  Let $E=\sum a_iE_i$ be the discrepancy divisor, i.e., a linear combination of
  exceptional divisors such that
  $$
  K_{X^{\rm m}}+\mu^{-1}_*\Delta \sim_{\bQ} \mu^*(K_X + \Delta) + E
  $$ and let
  $\Delta^{\rm m}\leteq \mu^{-1}_*\Delta + \sum_{a_i\leq -1}E_i$.  For an irreducible
  divisor $F$ on a birational model of $X$ we define its discrepancy as its
  coefficient in $E$.  Notice that as divisors correspond to valuations, this
  discrepancy is independent of the model chosen, it only depends on the divisor.  A
  \emph{non-klt place} of a pair $(X,\Delta)$ is an irreducible divisor $F$ over $X$
  with discrepancy at most $-1$ and a \emph{non-klt center} is the image of any
  non-klt place.  $\excnklt(\mu)$ denotes the union of the loci of all non-klt places
  of $\phi$.

Note that in the
  literature, non-klt places and centers are often called log canonical places and
  centers. For a more detailed and precise definition see
  \cite[p.37]{Hacon-Kovacs10}.
  
  Now if $(X^{\rm m}, \Delta^{\rm m})$ is as above, then it is a \emph{minimal dlt
    model} of $(X,\Delta)$ if it is a dlt pair and the discrepancy of every
  $\mu$-exceptional divisor is at most $-1$ cf.\ \cite{KK10}. Note that if
  $(X,\Delta)$ is lc with a minimal dlt model $(X^{\rm m}, \Delta^{\rm m})$, then
  $K_{X^{\rm m}}+\Delta^{\rm m} \sim_{\bQ} \mu^*(K_X+\Delta)$.

\subsection{Rational pairs}

Recall the definition of a \emph{log resolution} from \eqref{ssec:basic-definitions}:
A morphism of pairs $\phi:(Y,\pairD_Y)\to (X,\pairD)$ is a \emph{log resolution of
  $(X,\pairD)$} if $\phi:Y\to X$ is proper and birational,
$\pairD_Y=\phi^{-1}_*\pairD$, and $(\pairD_Y)_{\red}+\exc(\phi)$ is an snc divisor on
$Y$.

\begin{defn}\label{def:normal-pair}
  Let $(X,\pairD)$ be a pair and $\pairD$ an integral divisor. Then $(X,\pairD)$ is
  called a \emph{normal pair} if there exists a log resolution $\phi:(Y,\pairD_Y)\to
  (X,\pairD)$ such that the natural morphism $\phi^\#:\sO_X(- \pairD)\to
  \phi_*\sO_Y(-{\pairD_Y})$ is an isomorphism.
\end{defn}

\begin{defn}
  \label{def:weakly-rtl-sing-pairs}
  A pair $(X,\pairD)$ with $\pairD$ an integral divisor is called a \emph{weakly
    rational pair} if there is a log resolution $\phi:(Y,\pairD_Y)\to (X,\pairD)$
  such that the natural morphism $\sO_X(- \pairD)\to \myR\phi_*\sO_Y(-{\pairD_Y})$
  has a left inverse.
\end{defn}

\begin{lem}
  Let $(X,\pairD)$ be a weakly rational pair. Then it is a normal pair.
\end{lem}

\begin{proof}
  The $0^\text{th}$ cohomology of the left inverse of $\sO_X(- \pairD)\to
  \myR\phi_*\sO_Y(-{\pairD_Y})$ gives a left inverse of $\phi^\#:\sO_X(- \pairD)\to
  \phi_*\sO_Y(-{\pairD_Y})$. As the morphism $\phi$ is birational, the kernel of the
  left inverse of $\phi^\#$ is a torsion sheaf. However, since
  $\phi_*\sO_Y(-{\pairD_Y})$ is torsion-free, this implies that $\phi^\#$ is an
  isomorphism.
\end{proof}

\begin{defn}\cite{KollarKovacsRP}
  \label{def:rat-sing-pairs}
  Let $(X,\pairD)$ be a pair where $\pairD$ is an integral divisor. Then $(X,\pairD)$
  is called a \emph{rational pair} if there exists a log resolution
  $\phi:(Y,\pairD_Y)\to (X,\pairD)$ such that
  \begin{enumerate-cont}
  \item $\sO_X(-\pairD)\simeq \phi_*\sO_Y(-{\pairD_Y})$, i.e., $(X,\pairD)$ is
    normal,
  \item $\myR^i\phi_*\sO_Y(-{\pairD_Y})=0$ for $i>0$, and
  \item $\myR^i\phi_*\omega_Y({\pairD_Y})=0$ for $i>0$.
    \label{item:GR-vanishing-for-pairs}
  \end{enumerate-cont}
\end{defn}

\begin{lem}
  Let $(X,\pairD)$ be a pair where $\pairD$ is an integral divisor. Then it is a
  rational pair if and only if it is a weakly rational pair and
  $\myR^i\phi_*\omega_Y({\pairD_Y})=0$ for $i>0$.
\end{lem}

\begin{proof}
  This follows directly from \cite[105]{KollarKovacsRP}. 
\end{proof}

\begin{rem}\label{rem:rtl-is-relative}
  Note that the notion of a \emph{rational pair} describes the ``singularity'' of the
  relationship between $X$ and $\Delta$.  From the definition it is not clear for
  instance whether $(X,\Delta)$ being rational implies that $X$ has rational
  singularities.   
\end{rem}

\begin{rem}
  If $\Delta=\emptyset$, then
  (\ref{def:rat-sing-pairs}.\ref{item:GR-vanishing-for-pairs}) follows from
  Grauert-Riemenschneider vanishing and $X$ is weakly rational if and only if it is
  rational by \cite{Kovacs00b}. 
\end{rem}

\subsection{Generalized pairs}

\begin{defn}
  A \emph{generalized pair} $(X,\Sigma)$ consists of an equidimensional variety
  (i.e., a reduced scheme of finite type over a field $k$) $X$ and a subscheme
  $\Sigma\subseteq X$.  A morphism of generalized pairs $\phi:(Y,\Gamma)\to
  (X,\Sigma)$ is a morphism $\phi:Y\to X$ such that $\phi(\Gamma)\subseteq \Sigma$.
  A \emph{reduced generalized pair} is a generalized pair $(X,\Sigma)$ such that
  $\Sigma$ is reduced.

  The \emph{log resolution} of a generalized pair $(X, W)$ is a proper birational
  morphism $\pi: \wt X\to X$ such that $\exc(\pi)$ is a divisor and
  $\pi^{-1}W+\exc(\pi)$ is an snc divisor.

  Let $X$ be a complex scheme and $\Sigma$ a closed subscheme whose complement in $X$
  is dense. Then $(X_{\kdot}, \Sigma_\kdot)\to (X, \Sigma)$ is a \emph{good
    hyperresolution} if $X_\kdot\to X$ is a hyperresolution, and if
  $U_\kdot=X_\kdot\times_X (X\setminus \Sigma)$ and $\Sigma_\kdot=X_\kdot\setminus
  U_\kdot$, then, for all $\alpha$, either $\Sigma_\alpha$ is a divisor with normal
  crossings on $X_\alpha$ or $\Sigma_\alpha=X_\alpha$. Notice that it is possible
  that $X_\kdot$ has components that map into $\Sigma$. These component are contained
  in $\Sigma_\mydot$.  For more details and the existence of such hyperresolutions
  see \cite[6.2]{DuBois81} and \cite[IV.1.21, IV.1.25, IV.2.1]{GNPP88}.  For a primer
  on hyperresolutions see the appendix of \cite{Kovacs-Schwede09}.
\end{defn}

Let $(X,\Sigma)$ be a reduced generalized pair.  Consider the \uj{Deligne-Du~Bois} complex
of $(X,\Sigma)$ defined by Steenbrink \cite[\S 3]{Steenbrink85}:

\begin{defn}
  The \emph{\uj{Deligne-Du~Bois} complex} of the reduced generalized pair $(X,\Sigma)$ is
  the mapping cone of the natural morphism $\varrho:\Om^\kdot_X\to \Om^\kdot_\Sigma$
  twisted by $(-1)$.  In other words, it is an object $\Om^\kdot_{X,\Sigma}$ in
  $D_{\rm filt}(X)$ such that it completes $\varrho$ to an exact triangle:
  \begin{equation}
    \label{eq:11}
    \xymatrix{%
      \Om^\kdot_{X,\Sigma} \ar[r] & \Om^\kdot_{X} \ar[r] & \Om^\kdot_{\Sigma}
      \ar[r]^-{+1} & .  }
  \end{equation}
  The associated graded quotients of $\Om^\kdot_{X,\Sigma}$ will be denoted as usual:
  $$
  \Om^p_{X,\Sigma}\leteq Gr^p_{\rm filt}\Om^\kdot_{X,\Sigma}[p].
  $$
  Notice that the above triangle is in $D_{\rm filt}(X)$ and hence for all $p\in\bN$
  we obtain another exact triangle:
  \begin{equation}
    \label{eq:4}
    \xymatrix{%
      \Om^p_{X,\Sigma} \ar[r] &  \Om^p_{X} \ar[r] & \Om^p_{\Sigma}
      \ar[r]^-{+1} & . 
    }
  \end{equation}
\end{defn}

\begin{example}
  Let $(X,\Sigma)$ be an snc pair. Then $\Om^\kdot_{X,\Sigma}\qis
  \Omega_X^\kdot(\log\Sigma)(-\Sigma)$.
\end{example}

\noindent
The \uj{Deligne-Du~Bois} complex of a pair is funtorial in the following sense:

\begin{prop}
  Let $\phi:(Y,\Gamma)\to (X,\Delta)$ be a morphism of generalized pairs.  Then there
  exists a filtered natural morphism $\Om^\kdot_{X,\Sigma}\to
  \myR\phi_*\Om_{Y,\Gamma}^\kdot$.
\end{prop}

\begin{proof}
  There exist compatible filtered natural morphisms $\Om_X^\kdot\to
  \myR\phi_*\Om_Y^\kdot$ and $\Om_\Sigma^\kdot\to \myR\phi_*\Om_\Gamma^\kdot$ by
  (\ref{defDB}.\ref{functorial}). They induce the following morphism between exact
  triangles,
  $$
  \xymatrix{%
    \Om^0_{X,\Sigma} \ar[r]\ar@{-->}[d] & \Om^0_{X} \ar[r]\ar[d] & \Om^0_{\Sigma}
    \ar[r]^-{+1} \ar[d] & \\
    \myR\phi_*\Om^0_{Y,\Gamma} \ar[r] & \myR\phi_*\Om^0_{Y} \ar[r] &
    \myR\phi_*\Om^0_{\Gamma} \ar[r]^-{+1} & , }
  $$
  and thus one obtains the desired natural morphism.
\end{proof}

It follows easily from the definition and \eqref{prop:top-coh-vanishes} that we have
the following bounds on the non-zero cohomology sheaves of $\Om_{X,\Sigma}^p$.

\begin{prop}\label{prop:top-coh-vanishes-rel}
  Let $X$ be a positive dimensional variety. Then the $i^{\text{th}}$ cohomology
  sheaf of $\ul{\Omega}_{X,\Sigma}^p$ vanishes for all $i\geq \dim X$, i.e.,
  $h^i(\ul{\Omega}_{X,\Sigma}^p)=0$ for all $p$ and for all $i\geq \dim X$.
\end{prop}

\begin{proof}
  This follows directly from \eqref{prop:top-coh-vanishes} using the long exact
  cohomology sequence associated to (\ref{eq:4}).
\end{proof}

\subsection{DB pairs and the DB defect}

\begin{defn}\label{def:DB-defect}
  Recall the short exact sequence for the restriction of regular functions from $X$
  to $\Sigma$:
  $$
  \xymatrix{%
    0 \ar[r] & \sI_{\Sigma\subseteq X} \ar[r] & \sO_X \ar[r] & \sO_\Sigma \ar[r] & 0.
  }
  $$

  By (\ref{defDB}.\ref{item:dR-to-DB}) there exist compatible natural maps $\sO_X\to
  \Om_X^0$ and $\sO_\Sigma\to \Om_\Sigma^0$, and they induce a morphism between exact
  triangles,
  \begin{equation}
    \label{eq:12}
    \xymatrix{%
      \sI_{\Sigma\subseteq X} \ar[r]\ar@{-->}[d] & \sO_X \ar[r]\ar[d] & 
      \sO_\Sigma \ar[d]  \ar[r]^-{+1} & \\
      \Om^0_{X,\Sigma} \ar[r] & \Om^0_{X} \ar[r] & \Om^0_{\Sigma} \ar[r]^-{+1} & , }
  \end{equation}
  A reduced generalized pair $(X,\Sigma)$ will be called a \emph{DB pair} if the
  natural morphism $\sI_{\Sigma\subseteq X}\to \Om^0_{X,\Sigma}$ from (\ref{eq:12})
  is a quasi-isomorphism.
\end{defn}

\begin{rem}
  Note that just like the notion of a {rational pair}, the notion of a \emph{DB pair}
  describes the ``singularity'' of the relationship between $X$ and $\Sigma$.  From
  the definition it is not clear for instance whether $(X,\Sigma)$ being DB implies
  that $X$ has DB singularities. \ujj{I expect that this is
    not true, but at the time of writing this article I do not know an example of an
    irreducible $X$ and an appropriate $\Sigma\subset X$ such that $(X,\Sigma)$ is a
    DB pair, but $X$ does not have DB singularities.}
\end{rem}

\begin{prop}\label{prop:rel-funtorial}
  Let $\phi:(Y,\Gamma)\to (X,\Sigma)$ be a morphism of generalized pairs.  Then there
  exists a natural morphism $\Om^0_{X,\Sigma}\to \myR\phi_*\Om_{Y,\Gamma}^0$ and a
  commutative diagram,
  \begin{equation*}
    \xymatrix{%
      \sI_{\Sigma\subseteq X} \ar[d]\ar[r] & \Om^0_{X,\Sigma}\ar[d] \\
      \myR\phi_*\sI_{\Gamma\subseteq Y} \ar[r] & \myR\phi_*\Om_{Y,\Gamma}^0 \\
    }
  \end{equation*}
\end{prop}

\begin{proof}
  Similarly to (\ref{eq:12}) and one obtains a commutative diagram for $(Y,\Gamma)$:
  $$
  \xymatrix{%
    \sI_{\Gamma\subseteq Y} \ar[r]\ar[d] & \sO_Y \ar[r]\ar[d] & \sO_\Gamma \ar[d]
    \ar[r]^-{+1} & \\
    \Om^0_{Y,\Gamma} \ar[r] & \Om^0_{Y} \ar[r] & \Om^0_{\Gamma}
    \ar[r]^-{+1} & .
  }
  $$
  Then $\phi$ induces a morphism between these diagrams:
    
  $$
  \xymatrix{%
    \sI_{\Sigma\subseteq X} \ar[rr]\ar[dd] \ar[rd] && \sO_X \ar[rr]\ar'[d][dd]
    \ar[rd] && \sO_\Sigma \ar'[d][dd] \ar[rd]
    \ar[r]^-{+1} & \\
    & \Om^0_{X,\Sigma}  \ar[rr]\ar[dd] && \Om^0_{X}
    \ar[rr]\ar[dd] &&  \Om^0_{\Sigma}   \ar[dd] \ar[r]^-{+1} & \\
    \myR\phi_* \sI_{\Gamma\subseteq Y} \ar'[r][rr]\ar[rd] && 
    \myR\phi_* \sO_\Gamma \ar'[r][rr]\ar[rd] &&
    \myR\phi_* \sO_Y  \ar[rd]  \ar[r]^-{+1} & \\
    & \myR\phi_* \Om^0_{Y,\Gamma} \ar[rr] && \myR\phi_* \Om^0_{Y} \ar[rr] && 
    \myR\phi_* \Om^0_{\Gamma} \ar[r]^-{+1} &
    .  }
  $$
  The front face of this diagram provides the one claimed in the statement.
\end{proof}

Similarly to \eqref{def:db-defect} we introduce the {DB defect} of the pair
$(X,\Sigma)$: 

\begin{defn}\label{def:rel-db-defect} The
  \emph{DB defect} of the pair $(X,\Sigma)$ is the mapping cone of the morphism
  $\sI_{\Sigma\subseteq X}\to \Om^0_{X,\Sigma}$. It is denoted by
  $\Om_{X,\Sigma}^\times$.  Again, one has the exact triangles,
  \begin{gather}
    \xymatrix{%
      \sI_{\Sigma\subseteq X} \ar[r] & \Om^0_{X,\Sigma} \ar[r] &
      \Om^\times_{X,\Sigma} \ar[r]^-{+1} & .
    }\label{eq:5}
    \\
    \intertext{and} %
    \xymatrix{%
      \Om^\times_{X,\Sigma} \ar[r] & \Om^\times_{X} \ar[r] & \Om^\times_{\Sigma}
      \ar[r]^-{+1} & .  
    }\label{eq:6}
  \end{gather}
  And, again, one has that 
  \begin{equation}
    \label{eq:8}
    h^0(\Om_{X,\Sigma}^\times)\simeq
    h^0(\Om_{X,\Sigma}^0)/\sI_{\Sigma\subseteq X} \quad\text{and}\quad
    h^i(\Om_{X,\Sigma}^\times)\simeq h^i(\Om_{X,\Sigma}^0) \quad\text{for $i>0$.}
  \end{equation} 
\end{defn}

\begin{lem}
  \label{lem:db-defect-rel}
  Let $(X,\Sigma)$ be a reduced generalized pair. Then the following are equivalent: 
  \begin{enumerate}
  \item The pair $(X,\Sigma)$ is DB.
    \label{item:3}
  \item The DB defect of $(X,\Sigma)$ is acyclic, that is, $\Om_{X,\Sigma}^\times\qis
    0$.
    \label{item:6}
  \item The induced natural morphism $\ul{\Omega}_X^\times\to
    \ul{\Omega}_{\Sigma}^\times$ is a quasi-isomorphism.
    \label{item:7}
  \item The induced natural morphism $h^i(\ul{\Omega}_X^\times)\to
    h^i(\ul{\Omega}_{\Sigma}^\times)$ is an isomorphism for all $i\in\bZ$.
    \label{item:4}
  \item The induced natural morphism $h^i(\ul{\Omega}_X^0)\to
    h^i(\ul{\Omega}_{\Sigma}^0)$ is an isomorphism for all $i\neq 0$ and a surjection
    with kernel isomorphic to $\sI_{\Sigma\subseteq X}$ for $i=0$.
    \label{item:5}
  \end{enumerate}
\end{lem}

\begin{subrem}
  This statement also applies in the case when $\Sigma=\emptyset$, so it implies
  \eqref{lem:db-defect}.
\end{subrem}

\begin{proof}
  The equivalence of (\ref{lem:db-defect-rel}.\ref{item:3}) and
  (\ref{lem:db-defect-rel}.\ref{item:6}) follows from (\ref{eq:5}), the equivalence
  of (\ref{lem:db-defect-rel}.\ref{item:6}) and
  (\ref{lem:db-defect-rel}.\ref{item:7}) follows from (\ref{eq:6}), the equivalence
  of (\ref{lem:db-defect-rel}.\ref{item:7}) and
  (\ref{lem:db-defect-rel}.\ref{item:4}) follows from the definition of
  quasi-isomorphism, and the equivalence of (\ref{lem:db-defect-rel}.\ref{item:4})
  and (\ref{lem:db-defect-rel}.\ref{item:5}) follows from the definition of the DB
  defect $\Om_{X,\Sigma}^\times$ \eqref{def:rel-db-defect} and (\ref{eq:8}).
\end{proof}

Cutting by hyperplanes works the same way as in the absolute case:
\begin{prop}\label{prop:db-cx-of-hyper-rel}
  Let $(X,\Sigma)$ be a reduced general pair where $X$ is a quasi-projective variety
  and $H\subset X$ a general member of a very ample linear system. Then
  $\Om_{H,H\cap\Sigma}^\kdot\qis \Om_{X,\Sigma}^\kdot\otimes_L\sO_H$ and
  $\Om_{H,H\cap\Sigma}^\times\qis \Om_{X,\Sigma}^\times\otimes_L\sO_H$.
\end{prop}

\begin{proof}
  This follows directly from \eqref{prop:db-cx-of-hyper}, (\ref{eq:11}), and 
  \eqref{prop:hyper-plane-cuts}.
\end{proof}

We also have the following adjunction type statement.

\begin{prop}
  \label{prop:DB-defect-of-a-union}
  Let $X=(Y\cup Z)_{\red}$ be a union of closed reduced subschemes with intersection
  $W=(Y\cap Z)_{\red}$.  Then the DB defects of the pairs $(X,Y)$ and $(Z,W)$ are
  quasi-isomorphic.  I.e.,
  $$\Om^\times_{X,Y}\qis \Om^\times_{Z,W}.$$ 
\end{prop}

\begin{proof}
  Consider the following diagram of exact triangles,
  \begin{equation*}
    \xymatrix{%
      \text{\phantom{mn}}
      \Om^\times_{X,Y} \ar[d]^\alpha \ar[r] & \Om^\times_X \ar[d]^\beta \ar[r] &
      \Om^\times_Y \ar[d]^\gamma \ar[r]^-{+1} & \\ 
      \text{\phantom{mn}}\Om^\times_{Z,W}  \ar[r] & \Om^\times_Z  \ar[r] &
      \Om^\times_W  \ar[r]^-{+1} &, \\  
    }
  \end{equation*}
  where $\beta$ and $\gamma$ are the natural restriction morphisms and $\alpha$ is
  the morphism induced by $\beta$ and $\gamma$ on the mapping cones. Then by
  \cite[2.1]{KK10} there exists an exact triangle 
  \begin{equation*}
    \xymatrix{%
      \sfQ \ar[r] & \Om^\times_{Y} \oplus \Om^\times_{Z}
      \ar[r]  &  \Om^\times_W \ar[r]^-{+1} & .
    }
  \end{equation*}
  and a map $\sigma:\Om^\times_X\to \sfQ$ compatible with the above diagram such that
  $\alpha$ is an isomorphism if and only if $\sigma$ is one.  On the other hand,
  $\sigma$ is indeed an isomorphism by \eqref{prop:DB-defect-for-unions} and so the
  statement follows.
\end{proof}

\section{Cohomology with compact support}\label{sec:cohom-with-comp}

Let $X$ be a complex scheme of finite type and $\iota:\Sigma\into X$ a closed
subscheme.  Deligne's Hodge theory applied in this situation gives the following
theorem:

\begin{thm}\cite{MR0498552}%
  \label{thm:hodge} 
  Let $X$ be a complex scheme of finite type, $\iota:\Sigma\into X$ a closed
  subscheme and $j:U\leteq X\setminus \Sigma\into X$.  Then
  \begin{enumerate}
  \item The natural composition map $j_{!}\bC_{U}\to \sI_{\Sigma\subseteq X} \to
    \Om^\kdot_{X,\Sigma}$ is a quasi-isomorphism, i.e., $\Om^\kdot_{X,\Sigma}$ is a
    resolution of the sheaf $j_{!}\bC_{U}$.
  \item The natural map $H_{\rm c}^\kdot(U,\bC)\to \bH^\kdot(X, \Om^\kdot_{X,\Sigma})$ is
    an isomorphism.
  \item If in addition $X$ is proper, then the spectral sequence,
    $$
    E\uj{_1}^{p,q}= \bH^q(X, \Om^p_{X,\Sigma}) \Rightarrow H_{\rm c}^{p+q}(U,\bC)
    $$
    degenerates at $E\uj{_1}$ and abuts to the Hodge filtration of Deligne's mixed Hodge
    structure. \label{item:9}
  \end{enumerate}
\end{thm}

\begin{proof}
  Consider an embedded hyperresolution of $\Sigma\subseteq X$:
  $$
  \xymatrix{%
    \text{\phantom{$\kdot$}}\Sigma_\kdot \hskip-.5ex\ar[r]^-{\varrho_\kdot}
    \ar[d]_{\varepsilon_\kdot\hskip-.5ex}  & X_\kdot  \ar[d]^{\varepsilon_\kdot}
    \hskip-1ex \\ 
    \Sigma \ar[r]_\varrho & X
  }
  $$
  Then by (\ref{defDB}.\ref{item:1}) and by definition $\Om_{X,\Sigma}^\kdot \qis
  \myR{\varepsilon_\kdot}_* \Omega^\kdot_{X_\kdot, \Sigma_\kdot}$.  The statements
  then follow from \cite[8.1, 8.2, 9.3]{MR0498552}. See also \cite[IV.4]{GNPP88}.
\end{proof}

\begin{cor}\label{cor:surjectivity}
  Let $X$ be a proper complex scheme of finite type, $\iota:\Sigma\into X$ a closed
  subscheme and $j:U\leteq X\setminus \Sigma\into X$.  Then the natural map
  $$
  H^i(X, \sI_{\Sigma\subseteq X}) \to  \bH^i(X, \Om_{X,\Sigma}^0)
  $$
  is surjective for all $i\in \bN$.
\end{cor}

\begin{proof}
  By (\ref{thm:hodge}.\ref{item:9}) the natural composition map
  $$
  H^i_{\rm c}(U, \bC) \to H^i(X, \sI_{\Sigma\subseteq X}) \to \bH^i(X,
  \Om_{X,\Sigma}^0)
  $$
  is surjective. This clearly implies the statement.
\end{proof}

\section{DB pairs in nature}\label{sec:db-pairs}

\begin{prop}
  \label{prop:pair-of-DBs-is-DB}
  Let $(X,\Sigma)$ be a reduced generalized pair. 
  If either $X$ or $\Sigma$ is DB, then the other
  one is DB if and only if $(X,\Sigma)$ is a DB pair.
\end{prop}

\begin{proof}
  Consider the exact triangle (\ref{eq:6})
  $$
  \xymatrix{%
    \Om^\times_{X,\Sigma} \ar[r] & \Om^\times_{X} \ar[r] & \Om^\times_{\Sigma}
    \ar[r]^-{+1} & .  
  }
  $$
  Clearly, if one of the objects in this triangle is acyclic, then it is equivalent
  that the other two are acyclic. Then the statement follows by \eqref{lem:db-defect}
  and \eqref{lem:db-defect-rel}.
\end{proof}

As one expects it from a good notion of singularity, smooth points are DB. For pairs,
being smooth is replaced by being snc.

\begin{cor}
  \label{cor:snc-is-db}
  Let $(X,\Delta)$ be an snc pair. Then it is also a DB pair.
\end{cor}

\begin{proof}
  This follows directly from \eqref{prop:pair-of-DBs-is-DB} cf.\
  (\ref{defDB}.\ref{item:8}) \cite[3.2]{Steenbrink85}. It also follows from
  \eqref{cor:lc-is-DB}.
\end{proof}

\begin{cor}\label{cor:lc-is-DB}
  Let $(X,\pairD)$ be a log canonical pair and $\Lambda\subset X$ an effective
  integral Weil divisor such that $\supp\Lambda\subseteq \supp\rdown\pairD$. Then
  $(X,\Lambda)$ is a DB pair.
\end{cor}

\begin{proof}
  By choice $\Lambda$ is a union of non-klt centers of the pair $(X,\pairD)$ and
  hence by \cite[Theorem~1.4]{KK10} both $X$ and $\Lambda$ are DB.  Then
  $(X,\Lambda)$ is a DB pair by \eqref{prop:pair-of-DBs-is-DB}.
\end{proof}

\begin{thm}
  \label{thm:wdb-is-db}
  Let $(X,\Sigma)$ be a reduced generalized pair. Assume that the natural morphism
  $\sI_{\Sigma\subseteq X}\to\Om_{X,\Sigma}^0$ has a left inverse. Then $(X,\Sigma)$
  is a DB pair.
\end{thm}

\begin{proof}
  We will mimic the proof of \cite[1.5]{Kovacs00c}.  The statement is local so we may
  assume that $X$ is affine and hence quasi-projective
  \begin{lem}\label{lem:local-surjectivity}
    Assume that there exists a finite subset $P\subseteq X$ such that $(X\setminus
    P,\Sigma\setminus P)$ is a DB pair. Then the induced morphism
    $$
    H^i_P(X, \sI_{\Sigma\subseteq X}) \to  \bH^i_P(X, \Om_{X,\Sigma}^0)
    $$
    is surjective for all $i\in\bN$.
  \end{lem}
  \begin{proof}
    Let $\ol X$ be a projective closure of $X$ and let $\ol\Sigma$ be the closure of
    $\Sigma$ in $\ol X$. Let $Q=\ol X\setminus X$, $Z=P\disjoint Q$, and $U=\ol
    X\setminus Z=X\setminus P$. Consider the exact triangle of functors,
    \begin{equation}
      \label{eq:13}
      \xymatrix{%
        \bH^0_Z({\ol X},\blank ) \ar[r] & \bH^0({\ol X},\blank ) \ar[r] & \bH^0(U,\blank )
        \ar[r]^-{+1} & 
      }
    \end{equation}
    and apply it to the morphism $\sI_{{\ol \Sigma}\subseteq {\ol X}} \to \Om_{{\ol
        X},{\ol \Sigma}}^0$. One obtains a morphism of two long exact sequences:
    $$
    \hskip-1em\xymatrix{%
      \bH^{i-1}(U,\sI_{{\ol \Sigma}\subseteq {\ol X}} ) \ar[d]^{\alpha_{i-1}}\ar[r] &
      \bH^{i}_Z({\ol X},\sI_{{\ol \Sigma}\subseteq {\ol X}} ) \ar[d]^{\beta_i}\ar[r]
      & \bH^{i}({\ol X},\sI_{{\ol \Sigma}\subseteq {\ol X}} ) \ar[d]^{\gamma_i}\ar[r]
      &
      \bH^{i}(U,\sI_{{\ol \Sigma}\subseteq {\ol X}} ) \ar[d]^{\alpha_i}\\
      \bH^{i-1}(U,\Om_{{\ol X},{\ol \Sigma}}^0 ) \ar[r] & \bH^{i}_Z({\ol X},\Om_{{\ol
          X},{\ol \Sigma}}^0 ) \ar[r] & \bH^{i}({\ol X},\Om_{{\ol X},{\ol \Sigma}}^0
      ) \ar[r] & \bH^{i}(U,\Om_{{\ol X},{\ol \Sigma}}^0 )  .}
    $$
    By assumption, $\alpha_i$ is an isomorphism for all $i$. By
    \eqref{cor:surjectivity}, $\gamma_i$ is surjective for all $i$. Then by the
    5-lemma, $\beta_i$ is also surjective for all $i$.
    
    By construction $P\cap Q=\emptyset$ and hence
    \begin{align*}
      H^i_Z({\ol X}, \sI_{\ol\Sigma\subseteq {\ol X}}) &\simeq H^i_P({\ol X},
      \sI_{\ol\Sigma\subseteq {\ol X}}) \oplus
      H^i_Q({\ol X}, \sI_{\ol\Sigma\subseteq {\ol X}}) \\
      \bH^i_Z({\ol X}, \Om_{{\ol X},\ol\Sigma}^0) &\simeq \bH^i_P({\ol X}, \Om_{{\ol
          X},\ol\Sigma}^0) \oplus \bH^i_Q({\ol X}, \Om_{{\ol X},\ol\Sigma}^0)
    \end{align*}
    It follows that the natural map (which is also the restriction of $\beta_i$), 
    $$
    H^i_P(\ol X, \sI_{\ol\Sigma\subseteq \ol X}) \to  \bH^i_P(\ol X, \Om_{\ol
      X,\ol\Sigma}^0) 
    $$
    is surjective for all $i$

    Now, by excision on local cohomology one has that
    $$
    H^i_P(\ol X, \sI_{\ol\Sigma\subseteq \ol X}) \simeq H^i_P(X, \sI_{\Sigma\subseteq X})
    \quad\text{and}\quad 
    \bH^i_P(\ol X, \Om_{\ol X,\ol\Sigma}^0) \simeq \bH^i_P(X, \Om_{X,\Sigma}^0).
    $$
    and so \eqref{lem:local-surjectivity} follows.
  \end{proof}

  It is now relatively straightforward to finish the proof of \ref{thm:wdb-is-db}: 

  By taking repeated hyperplane sections and using \eqref{prop:db-cx-of-hyper-rel} we
  may assume that there exists a finite subset $P\subseteq X$ such that $(X\setminus
  P,\Sigma\setminus P)$ is a DB pair. Therefore we may apply \eqref{lem:local-surjectivity}.

  By assumption, the natural morphism $\sI_{\Sigma\subseteq X}\to\Om_{X,\Sigma}^0$
  has a left inverse. This implies that applying any cohomology operator on this map
  induces an injective map on cohomology. In particular, this implies that the
  natural morphism
  $$
  H^i_P(X, \sI_{\Sigma\subseteq X}) \to  \bH^i_P(X, \Om_{X,\Sigma}^0)
  $$
  is injective for all $i\in\bN$. By \eqref{lem:local-surjectivity} they are also
  surjective and hence an isomorphism. Thus the DB defect $\Om_{X,\Sigma}^\times$ is
  such that all of its local cohomology groups are zero:
  $$
  \bH^i_P(X,\Om_{X,\Sigma}^\times)=0\quad\text{for all $i$.}
  $$
  On the other hand, by assumption $\Om_{X,\Sigma}^\times$ is supported entirely on
  $P$, so $\bH^i(X\setminus P, \Om_{X,\Sigma}^\times)=0$ as well. However, then
  $\bH^i(X, \Om_{X,\Sigma}^\times)=0$ by the long exact sequence induced by
  (\ref{eq:13}). Now $\dim P\leq 0$ so the spectral sequence that computes
  hypercohomology from the sheaf cohomology of the cohomology of the complex
  $\Om_{X,\Sigma}^\times$ degenerates and gives that for any $i\in \bN$, $\bH^i(X,
  \Om_{X,\Sigma}^\times)=H^0(X, h^i(\Om_{X,\Sigma}^\times))$, so, since we assumed
  that $X$ is affine, it follows that $h^i(\Om_{X,\Sigma}^\times)=0$ for all $i$.
  Therefore $\Om_{X,\Sigma}^\times\qis 0$ and thus the statement is proven.
\end{proof}

\begin{cor}
  \label{corr:rtl-is-DB}
  Let $(X,\pairD)$ be a weakly rational pair. Then it is a DB pair.
\end{cor}

\begin{proof}
  Let $\phi:(Y,\pairD_Y)\to (X,\pairD)$ be a log resolution \ujjj{such that $\gamma$
    admits a left inverse $\gamma^\ell$}.  Then by \eqref{prop:rel-funtorial} one has
  the commutative diagram:
  \begin{equation*}
    \xymatrix{%
      \sO_X(-\Delta) \ar[d]^\gamma\ar[r] & \Om^0_{X,\Delta}\ar[d]^\alpha \\
      \myR\phi_*\sO_Y(-\Delta_Y) \ar[r]^-\delta_-\qis \ar@/^/[u]^{\gamma^\ell}
      & \myR\phi_*\Om_{Y,\Delta_Y}^0 &
    }
  \end{equation*}
  Recall that as $(Y,\pairD_Y)$ is an snc pair, it is also DB by
  \eqref{cor:snc-is-db} and hence $\delta$ is a quasi-isomorphism.
  \ujj{By assumption $(Y,\pairD_Y)$ is a weakly rational pair
    so $\gamma$ admits a left inverse $\gamma^\ell$.} Then $\gamma^\ell\circ
  \delta^{-1}\circ\alpha$ is a left inverse to $\sO_X(-\Delta)\to \Om^0_{X,\Delta}$,
  so the statement follows from \eqref{thm:wdb-is-db}.
\end{proof}

\begin{cor}
  \label{cor:rtl-is-DB}
  A rational pair is a DB pair.
\end{cor}

\begin{proof}
  As a rational pair is also a weakly rational pair, this is straighforward from
  \eqref{corr:rtl-is-DB}.
\end{proof}

\begin{cor}
  \label{cor:dlt-is-DB}
  Let $(X,\pairD)$ be a dlt pair and $\Lambda\subset X$ an effective integral Weil
  divisor such that $\supp\Lambda\subseteq \supp\rdown\pairD$. Then $(X,\Lambda)$ is
  a DB pair.
\end{cor}

\begin{proof}
  A dlt pair is also an lc pair, so this follows from \eqref{cor:lc-is-DB}.
\end{proof}

\ujj{
\begin{proof}[Proof \#2]
  If $(X,\pairD)$ is a dlt pair, then $(X,\Lambda)$ is a rational pair by
  \cite[111]{KollarKovacsRP}, 
  so this also follows from \eqref{cor:rtl-is-DB}.
\end{proof}
}

\section{Vanishing Theorems}\label{sec:vanishing-theorems}

\noindent
The folowing is the main vanishing result of this paper.  Note that a weaker version
of it appeared in \cite[13.4]{GKKP10}. 

\begin{thm}\label{thm:vanishing}
  Let $(X, \Sigma)$ be a DB pair and $\pi : \widetilde X \to X$ a proper birational
  morphism with $E := \exc(\pi)$. Let $\wt\Sigma = E \cup \pi^{-1}\Sigma$ and
  $\Upsilon := \ol{\pi(E) \setminus \Sigma}$, both considered with their induced
  reduced subscheme structure.  Further let $s\in\bN$, $s>0$ such that
  $h^i(\Om_{\Upsilon,\Upsilon\cap\Sigma}^\circ)=0$ for $i\geq s$.  Then
  $$ 
  \myR^{i}\pi_* \Om_{\wt X,\wt\Sigma}^0 = 0 \text{\qquad for all \, $i \geq s$.}
  $$
\end{thm}

\begin{proof}
  Let $\Gamma = \Sigma \cup \Upsilon$ and consider the exact triangle
  (\ref{defDB}.\ref{item:exact-triangle}),
  \begin{equation}
    \label{eq:2}
    \xymatrix{%
      \ul{\Omega}_X^0 \ar[r] & \ul{\Omega}_\Gamma^0 \oplus
      \myR\pi_*\ul{\Omega}_{\wt X}^0 \ar[r] & \myR\pi_*\ul{\Omega}_{\wt \Sigma}^0 
      \ar^-{+1}[r] & },
  \end{equation}
  which induces the long exact sequence of sheaves:
  \begin{equation*}
    \xymatrix{%
      \ar[r] & h^i(\ul{\Omega}_X^0)\ar[r]^-{(\alpha^i,\sigma^i)} &
      h^i(\ul{\Omega}_\Gamma^0)\oplus 
      \myR^i\pi_*\ul{\Omega}_{\wt X}^0  \ar[r]
      & 
      \myR^i\pi_*\ul{\Omega}_{\wt \Sigma}^0 \ar[r] &  
      h^{i+1}(\ul{\Omega}_X^0)\ar[r] & 
    }
  \end{equation*}

  By \eqref{lem:db-defect-rel} the natural morphism $\gamma^i:
  h^i(\ul{\Omega}_X^0)\to h^i(\ul{\Omega}_{\Sigma}^0)$ is an isomorphism for $i>0$.
  By \eqref{prop:DB-defect-of-a-union} and the assumption
  we obtain that $h^i(\Om_{\Gamma,\Sigma}^\circ)=0$ for $i\geq s>0$ and hence the
  natural morphism $\beta^i: h^i(\ul{\Omega}_\Gamma^0)\to
  h^i(\ul{\Omega}_{\Sigma}^0)$ is an isomorphism 
  for $i\geq s$.
  Using 
  the fact that $\gamma^i=\beta^i\circ\alpha^i$ we obtain that the morphism
  $\alpha^i: h^i(\ul{\Omega}_X^0) \to h^i(\ul{\Omega}_\Gamma^0)$ is an isomorphism
  for $i\geq s>0$ and hence the natural restriction map
  $$
  \varrho^i: \myR^i\pi_*\Om^0_{\wt X} \to \myR^i\pi_*\Om^0_{\wt \Sigma}
  $$
  is an isomorphism for $i\geq s$. This in turn implies that $\myR^i\pi_* \Om_{\wt
    X,\wt\Sigma}^0=0$ for $i\geq s$ as desired.
\end{proof}

\noindent
As a corollary, a slight generalization of \cite[13.4]{GKKP10} follows.

\begin{cor}\label{cor:vanishing-1}
  Let $(X, \Sigma)$ be a DB pair and $\pi : \widetilde X \to X$ a log resolution of
  $(X,\Sigma)$ with $E := \exc(\pi)$. Let $\wt\Sigma = E \cup \pi^{-1}\Sigma$ and
  $\Upsilon := \ol{\pi(E) \setminus \Sigma}$, both considered with their induced
  reduced subscheme structure.  Then
  $$ 
  \myR^{i}\pi_* \sI_{\wt\Sigma\subseteq \wt X} = 0 
  \text{\qquad for all \, $i \geq \max \bigl(\dim \Upsilon,1 \bigr)$.}
  $$
  In particular, if $X$ is normal of dimension $n\geq 2$, then $\myR^{n-1}\pi_*
  \sI_{\wt\Sigma\subseteq \wt X} = 0$.
\end{cor}

\begin{proof}
  Let $s= \max \bigl(\dim \Upsilon, 1 \bigr)$. Then
  $h^i(\Om_{\Upsilon,\Upsilon\cap\Sigma}^\circ)=0$ for $i\geq s$ by
  \eqref{prop:top-coh-vanishes-rel}.  As the pair $(\wt X,\wt \Sigma)$ is snc, it is
  also DB and hence $\Om^0_{\wt X,\wt \Sigma}\qis \sI_{\wt\Sigma\subseteq \wt X}$.
  Therefore the statement follows from \eqref{thm:vanishing}.
\end{proof}

\noindent
We have a stronger result for log canonical pairs \ujjj{and for that we need the
  following definition:

\begin{defn}\label{def:szabo-res}
  A log resolution of a dlt pair $(Z,\Theta)$, $g:(Y, \Gamma)\to (Z,\Theta)$ is
  called a \emph{Szab\'o-resolution}, if there exist $A,B$ effective $\bQ$-divisors
  on $Y$ without common irreducible components, such that $\supp (A+B)\subset \exc
  (g)$, $\rdown{A}=0$, and
  $$
  K_{Y}+\Gamma \sim_{\bQ} g^*(K_Z+\Theta) - A + B.
  $$
\end{defn}

\begin{rem}
  Every dlt pair admits a Szab\'o-resolution by \cite{MR1322695} (cf.\ \cite[2.44]{KM98}).
\end{rem}
}

\begin{cor}\label{cor:top-vanishing}
  Let $(X, \pairD)$ be a \uj{$\bQ$-factorial} log canonical pair, $\pi : \widetilde X
  \to X$ a log resolution of $(X, \pairD)$,
  \ujjj{and} 
  let $\wt \pairD=\bigl(\pi^{-1}_*\rdown \pairD + \excnklt(\pi)\bigr)_{\red}$.  Then
  \begin{equation*}
    \myR^{i}\pi_* \, \sO_{\widetilde X}(- \widetilde \pairD) = 0\quad\text{for $i>0$.}\\
  \end{equation*}
\end{cor}
\begin{proof}
  \ujjj{%
    First note that the statement is true if $(X,\Delta)$ is an snc pair and $\pi$ is
    the blow up of $X$ along a smooth center. Indeed, if the center is a non-klt
    center, then the statement is a direct consequence of the Kawamata-Viehweg
    vanishing theorem and if the center is not a non-klt center, then this is a
    Szab\'o-resolution and the statement follows as in the proof of
    \cite[111]{KollarKovacsRP}. This implies the following:

    \begin{sublem}\label{eq:15}
      Let $\pi_i:(X_i,\pairD_i)\to (X,\pairD)$ for $i=1,2$ be two log resolutions of
      $(X,\pairD)$ and let $\wt \pairD_i=\bigl((\pi_i^{-1})_*\rdown \pairD +
      \excnklt(\pi_i)\bigr)_{\red}\subset X_i$. Then
      \begin{equation*}
        \myR(\pi_1)_* \, \sO_{X_1}(-\widetilde \pairD_1) \simeq         
        \myR(\pi_2)_* \, \sO_{X_2}(-\widetilde \pairD_2)  
      \end{equation*}
    \end{sublem}
    \begin{proof}
      By \cite[Theorem 0.3.1(6)]{MR1896232} (cf.\ \cite[Theorem 3.8]{MR2180406}) the
      induced birational map between $X_1$ and $X_2$ can be written as a sequence of
      blowing ups and blowing downs along smooth centers.  Then the statement follows
      from the above observation and the definition of the $\wt \pairD$'s.
    \end{proof}

  }

  \ujjj{Now we turn to proving the general case.}
  Consider a minimal dlt model $\mu: (X^{\rm m},\pairD^{\rm m})\to (X, \pairD)$
  \cite[3.1]{KK10}. Let $\Sigma\leteq \rdown\pairD\cup \mu(\exc(\mu))$ considered
  with the induced reduced subscheme structure. From the definition of a minimal dlt
  model it follows that $\Sigma$ is a union of non-klt centers of $(X,\pairD)$. Then
  by \cite[Theorem~1.4]{KK10} both $X$ and $\Sigma$ are DB, and hence $(X,\Sigma)$ is
  a DB pair by \eqref{prop:pair-of-DBs-is-DB}.

  Since $(X^{\rm m},\pairD^{\rm m})$ is dlt, $(X^{\rm m},\rdown{\pairD^{\rm m}})$ is
  a DB pair by \eqref{cor:lc-is-DB}.
  Therefore,
  $$
  \Om_{X^{\rm m},\rdown{\pairD^{\rm m}}}^0\qis \sO_{X^{\rm m}}(-\rdown{\pairD^{\rm m}}).
  $$
  By the definition of a minimal dlt model $\rdown{\pairD^{\rm
      m}}=\bigl(\pi^{-1}\Sigma\bigr)_{\red}\supseteq \exc(\mu)$ and then it follows
  from \eqref{thm:vanishing} that $\myR^i\mu_*\sO_{X^{\rm m}}(-\rdown{\pairD^{\rm
      m}})=0$ for $i>0$ and hence
  \begin{equation}
    \label{eq:10}
    \myR\mu_*\sO_{X^{\rm m}}(-\rdown{\pairD^{\rm m}}) \qis \sO_{X}(-\rdown\pairD). 
  \end{equation}
  \ujj{%
    Next let $\sigma : \what X\to X$ be a log resolution of $(X,\pairD)$ that factors
    through both $\pi$ and $\mu$. Then one has the following commutative diagram:
    $$
    \xymatrix{%
      \what X \ar[r]^\tau \ar@{-->}[d]_\lambda \ar[rd]^\sigma &
      \text{\vphantom{$\what X$}}{X^{\rm m}\hskip-1.2ex} \ar[d]^\mu \\
      \wt X \ar[r]_\pi & \text{\vphantom{$\wt X$}}X
    }
    $$
    Let $\what \Delta\subseteq \what X$ denote the strict transform of $\wt \Delta$
    on $\what X$. It follows from the definition of a minimal dlt model that $\what
    \pairD$ is also the strict transform of $\rdown{\pairD^{\rm m}}$ from $X^{\rm
      m}$.
    
    As both $(X^{\rm m}, \pairD^{\rm m})$ and $(\wt X,\wt \pairD)$ are dlt, $(X^{\rm
      m},\rdown{\pairD^{\rm m}})$ and $(\wt X,\wt \pairD)$ are also rational by
    \cite[111]{KollarKovacsRP}, 
    so we have that
    \begin{align*}
      \myR\tau_*\sO_{\what X}(-\what\pairD) &\qis \sO_{X^{\rm m}}(-\rdown{\pairD^{\rm
          m}}),     \text{  and}\\ 
      \myR\lambda_*\sO_{\what X}(-\what\pairD) &\qis \sO_{\wt X}(-\wt\pairD)
    \end{align*}
    Therefore, by (\ref{eq:10}),
    \begin{multline*}
      \myR\pi_*\sO_{\wt X}(-\wt\pairD) \qis \myR\pi_*\myR\lambda_*\sO_{\what
        X}(-\what\pairD) \qis \myR\sigma_*\sO_{\what X}(-\what\pairD) \qis \\
      \myR\mu_*\myR\tau_*\sO_{\what X}(-\what\pairD) \qis \myR\mu_*\sO_{X^{\rm
          m}}(-\rdown{\pairD^{\rm m}}) \qis \sO_{X}(-\rdown\pairD).
    \end{multline*}
  }%
  \ujjj{%
    Next let $\tau : \what X\to X^{\rm m}$ be the Szab\'o-resolution of $(X^{\rm
      m},\pairD^{\rm m})$, $\sigma=\mu\circ\tau$, $\what \pairD=\tau^{-1}_*\rdown
    {\pairD^{\rm m}}=\bigl(\sigma^{-1}_*\rdown {\pairD} +
    \excnklt(\sigma)\bigr)_{\red} $, and $\lambda=\pi^{-1}\circ\sigma$.
    Then by \cite[111]{KollarKovacsRP}, 
    we have that
    $$
    \myR\tau_*\sO_{\what X}(-\what\pairD) \qis \sO_{X^{\rm m}}(-\rdown{\pairD^{\rm
        m}}),
    $$
    and hence, by (\ref{eq:10}),
    \begin{multline*}
      \myR\sigma_*\sO_{\what X}(-\what\pairD) \qis \myR\mu_*\myR\tau_*\sO_{\what
        X}(-\what\pairD) \qis\\ \myR\mu_*\sO_{X^{\rm m}}(-\rdown{\pairD^{\rm m}})
      \qis \sO_{X}(-\rdown\pairD).
    \end{multline*}
    The proof is finished by applying (\ref{eq:15}) to $\what X$ and $\wt X$.  }
\end{proof}

\noindent
Finally, observe that \ujjj{\eqref{cor:top-vanishing}} implies that log canonical
singularities are not too far from being rational:

\begin{cor}
  Let $X$ be a variety with log canonical singularities and $\pi:\wt X\to X$ a
  resolution of $X$ with $E_{\lc}\leteq \excnklt(\pi)$. Then 
  $$
  \myR^{i}\pi_* \, \sO_{\widetilde X}(-  E_{\lc}) = 0\quad\text{for $i>0$.}
  $$
\end{cor}


\begin{thebibliography}{GKKP10}

\ujjj{
\bibitem[AKMW02]{MR1896232}
{\sc D.~Abramovich, K.~Karu, K.~Matsuki, and J.~W{\l}odarczyk}:
  \emph{Torification and factorization of birational maps}, J. Amer. Math. Soc.
  \textbf{15} (2002), no.~3, 531--572 (electronic). {\sf\scriptsize MR1896232
  (2003c:14016)}

\bibitem[BL05]{MR2180406}
{\sc L.~Borisov and A.~Libgober}: \emph{Mc{K}ay correspondence for elliptic
  genera}, Ann. of Math. (2) \textbf{161} (2005), no.~3, 1521--1569.
  {\sf\scriptsize 2180406 (2008b:58030)}
}

\bibitem[Car85]{Carlson85}
{\sc J.~A. Carlson}: \emph{Polyhedral resolutions of algebraic varieties},
  Trans. Amer. Math. Soc. \textbf{292} (1985), no.~2, 595--612. {\sf\scriptsize
  MR808740 (87i:14008)}

\bibitem[Del74]{MR0498552}
{\sc P.~Deligne}: \emph{Th\'eorie de {H}odge. {III}}, Inst. Hautes \'Etudes
  Sci. Publ. Math. (1974), no.~44, 5--77. {\sf\scriptsize MR0498552 (58
  \#16653b)}

\bibitem[DB81]{DuBois81}
{\sc P.~{D}u {B}ois}: \emph{Complexe de de {R}ham f{i}ltr\'e d'une vari\'et\'e
  singuli\`ere}, Bull. Soc. Math. France \textbf{109} (1981), no.~1, 41--81.
  {\sf\scriptsize MR613848 (82j:14006)}

\bibitem[DJ74]{MR0376678}
{\sc P.~{D}{u {B}ois} and P.~Jarraud}: \emph{Une propri\'et\'e de commutation
  au changement de base des images directes sup\'erieures du faisceau
  structural}, C. R. Acad. Sci. Paris S\'er. A \textbf{279} (1974), 745--747.
  {\sf\scriptsize MR0376678 (51 \#12853)}

\bibitem[GKKP10]{GKKP10}
{\sc D.~Greb, S.~Kebekus, S.~J. Kov\'acs, and T.~Peternell}: \emph{Differential
  forms on log canonical spaces}, preprint, 2010. {\sf\scriptsize
  arXiv:1003.2913v3 [math.AG]}

\bibitem[GNPP88]{GNPP88}
{\sc F.~Guill{\'e}n, V.~Navarro{\ }Aznar, P.~Pascual{\ }Gainza, and F.~Puerta}:
  \emph{Hyperr\'esolutions cubiques et descente cohomologique}, Lecture Notes
  in Mathematics, vol. 1335, Springer-Verlag, Berlin, 1988, Papers from the
  Seminar on Hodge-Deligne Theory held in Barcelona, 1982. {\sf\scriptsize
  MR972983 (90a:14024)}

\bibitem[HK10]{Hacon-Kovacs10}
{\sc C.~D. Hacon and S.~J. Kov{\'a}cs}: \emph{Classification of higher
  dimensional algebraic varieties}, Oberwolfach Seminars, Birkh\"auser Boston,
  Boston, MA, 2010.

\bibitem[Kol95]{Kollar95s}
{\sc J.~Koll{\'a}r}: \emph{Shafarevich maps and automorphic forms}, M. B.
  Porter Lectures, Princeton University Press, Princeton, NJ, 1995.
  {\sf\scriptsize MR1341589 (96i:14016)}

\bibitem[KK09]{KollarKovacsRP}
{\sc J.~Koll{\'a}r and S.~J. Kov{\'a}cs}: \emph{Rational pairs}, preprint,
  2009.

\bibitem[KK10]{KK10}
{\sc J.~Koll{\'a}r and S.~J. Kov{\'a}cs}: \emph{Log canonical singularities are
  {D}u~{B}ois}, J. Amer. Math. Soc. \textbf{23} (2010), no.~3, 791--813.
  {\sf\scriptsize doi:10.1090/S0894-0347-10-00663-6}

\bibitem[KM98]{KM98}
{\sc J.~Koll{\'a}r and S.~Mori}: \emph{Birational geometry of algebraic
  varieties}, Cambridge Tracts in Mathematics, vol. 134, Cambridge University
  Press, Cambridge, 1998, With the collaboration of C. H. Clemens and A. Corti,
  Translated from the 1998 Japanese original. {\sf\scriptsize MR1658959
  (2000b:14018)}

\bibitem[Kov99]{Kovacs99}
{\sc S.~J. Kov{\'a}cs}: \emph{Rational, log canonical, {D}u {B}ois
  singularities: on the conjectures of {K}oll\'ar and {S}teenbrink}, Compositio
  Math. \textbf{118} (1999), no.~2, 123--133. {\sf\scriptsize MR1713307
  (2001g:14022)}

\bibitem[Kov00a]{Kovacs00b}
{\sc S.~J. Kov{\'a}cs}: \emph{A characterization of rational singularities},
  Duke Math. J. \textbf{102} (2000), no.~2, 187--191. {\sf\scriptsize MR1749436
  (2002b:14005)}

\bibitem[Kov00b]{Kovacs00c}
{\sc S.~J. Kov{\'a}cs}: \emph{Rational, log canonical, {D}u {B}ois
  singularities. {II}. {K}odaira vanishing and small deformations}, Compositio
  Math. \textbf{121} (2000), no.~3, 297--304. {\sf\scriptsize MR1761628
  (2001m:14028)}

\bibitem[KS09]{Kovacs-Schwede09}
{\sc S.~J. Kov{\'a}cs and K.~Schwede}: \emph{Hodge theory meets the minimal
  model program: a survey of log canonical and {{D}u~{B}ois} singularities},
  preprint, 2009. {\sf\scriptsize arXiv:0909.0993v1 [math.AG]}

\bibitem[KSS10]{KSS10}
{\sc S.~J. Kov{\'a}cs, K.~Schwede, and K.~E. Smith}: \emph{The canonical sheaf
  of {D}u {B}ois singularities}, Advances in Mathematics \textbf{224} (2010),
  no.~4, 1618--1640.

\bibitem[PS08]{PetersSteenbrinkBook}
{\sc C.~A.~M. Peters and J.~H.~M. Steenbrink}: \emph{Mixed {H}odge structures},
  Ergebnisse der Mathematik und ihrer Grenzgebiete. 3. Folge. A Series of
  Modern Surveys in Mathematics [Results in Mathematics and Related Areas. 3rd
  Series. A Series of Modern Surveys in Mathematics], vol.~52, Springer-Verlag,
  Berlin, 2008. {\sf\scriptsize MR2393625}

\bibitem[Sai00]{MR1741272}
{\sc M.~Saito}: \emph{Mixed {H}odge complexes on algebraic varieties}, Math.
  Ann. \textbf{316} (2000), no.~2, 283--331. {\sf\scriptsize MR1741272
  (2002h:14012)}

\bibitem[Sch06]{Schwede06}
{\sc K.~Schwede}: \emph{On {D}u {B}ois and {F}-injective singularities},
  Ph.D.~thesis, University of Washington, 2006.

\bibitem[Sch07]{MR2339829}
{\sc K.~Schwede}: \emph{A simple characterization of {D}u {B}ois
  singularities}, Compos. Math. \textbf{143} (2007), no.~4, 813--828.
  {\sf\scriptsize MR2339829 (2008k:14034)}

\bibitem[Sch09]{MR2503989}
{\sc K.~Schwede}: \emph{{$F$}-injective singularities are {D}u {B}ois}, Amer.
  J. Math. \textbf{131} (2009), no.~2, 445--473. {\sf\scriptsize MR2503989}

\bibitem[ST08]{MR2492473}
{\sc K.~Schwede and S.~Takagi}: \emph{Rational singularities associated to
  pairs}, Michigan Math. J. \textbf{57} (2008), 625--658, Special volume in
  honor of Melvin Hochster. {\sf\scriptsize MR2492473}

\bibitem[Ste85]{Steenbrink85}
{\sc J.~H.~M. Steenbrink}: \emph{Vanishing theorems on singular spaces},
  Ast\'erisque (1985), no.~130, 330--341, Differential systems and
  singularities (Luminy, 1983). {\sf\scriptsize MR804061 (87j:14026)}

\bibitem[Ste83]{SteenbrinkMixed}
{\sc J.~H.~M. Steenbrink}: \emph{Mixed {H}odge structures associated with
  isolated singularities}, Singularities, Part 2 (Arcata, Calif., 1981), Proc.
  Sympos. Pure Math., vol.~40, Amer. Math. Soc., Providence, RI, 1983,
  pp.~513--536. {\sf\scriptsize MR713277 (85d:32044)}

\ujjj{
\bibitem[Sza94]{MR1322695}
{\sc E.~Szab{\'o}}: \emph{Divisorial log terminal singularities}, J. Math. Sci.
  Univ. Tokyo \textbf{1} (1994), no.~3, 631--639. {\sf\scriptsize MR1322695
  (96f:14019)}
}

\end{thebibliography}

\def\cprime{$'$} \def\polhk#1{\setbox0=\hbox{#1}{\ooalign{\hidewidth
  \lower1.5ex\hbox{`}\hidewidth\crcr\unhbox0}}} \def\cprime{$'$}
  \def\cprime{$'$} \def\cprime{$'$} \def\cprime{$'$}
  \def\polhk#1{\setbox0=\hbox{#1}{\ooalign{\hidewidth
  \lower1.5ex\hbox{`}\hidewidth\crcr\unhbox0}}} \def\cdprime{$''$}
  \def\cprime{$'$} \def\cprime{$'$} \def\cprime{$'$} \def\cprime{$'$}
\providecommand{\bysame}{\leavevmode\hbox to3em{\hrulefill}\thinspace}
\providecommand{\MR}{\relax\ifhmode\unskip\space\fi MR}
\providecommand{\MRhref}[2]{%
  \href{http://www.ams.org/mathscinet-getitem?mr=#1}{#2}
}
\providecommand{\href}[2]{#2}

\end{document}